\documentclass{amsart}[10pt]
\usepackage[utf8]{inputenc}

\usepackage{url}
\usepackage{xcolor}
\usepackage{xfrac}
\usepackage{wrapfig}
\usepackage{amsthm}
\usepackage{amsmath}
\usepackage{amssymb}
\usepackage{bbm}
\usepackage{physics}
\usepackage{mathtools}
\usepackage{dsfont}
\usepackage{enumitem}
\usepackage{centernot}
\usepackage{mathtools}
\usepackage{stmaryrd}
\usepackage{faktor}
\usepackage{appendix}
\usepackage{graphicx}
\usepackage{dirtytalk}
\usepackage{circuitikz}
\usepackage[nice]{nicefrac}
\usepackage{caption}
\usepackage{subcaption}

\makeatletter
\newcommand{\xMapsto}[2][]{\ext@arrow 0599{\Mapstofill@}{#1}{#2}}
\def\Mapstofill@{\arrowfill@{\Mapstochar\Relbar}\Relbar\Rightarrow}
\makeatother

\newcommand{\N}{\mathbb{N}}
\newcommand{\Z}{\mathbb{Z}}

\newcommand{\R}{\mathbb{R}}
\newcommand{\C}{\mathbb{C}}
\newcommand{\T}{\mathbb{T}}
\newcommand{\HH}{\mathbb{H}}
\newcommand{\Rd}{\mathbb{R}^{d}}

\newcommand{\cA}{\mathcal{A}}

\newcommand{\cC}{\mathcal{C}}

\newcommand{\cF}{\mathcal{F}}

\newcommand{\cL}{\mathcal{L}}

\newcommand{\cR}{\mathcal{R}}

\newcommand{\cU}{\mathcal{U}}

\newcommand{\slz}{\text{SL}_{2}\left(\mathbb{Z}\right)}

\newcommand{\GL}{\text{GL}}

\def\ord{\text{ord}}
\def\covol{\text{covol}}

\newcommand{\ind}{\mathbbm{1}}

\newcommand{\pr}[1]{{\left(#1\right)}}
\newcommand{\prb}[1]{{\left[#1\right]}}
\newcommand{\prs}[1]{{\left\{#1\right\}}}
\newcommand{\ipp}[1]{{\left\langle #1\right\rangle}}
\newcommand{\floor}[1]{{\left\lfloor #1\right\rfloor}}
\newcommand{\ceil}[1]{{\left\lceil #1\right\rceil}}
\newcommand{\msum}[2]{\sum_{\begin{subarray}{c} 
   #1 \\ 
   #2 
 \end{subarray} }}

\newcommand{\sub}{\subset}

\numberwithin{equation}{section}

\newtheorem{thm}{Theorem}[section]

\newtheorem{corr}[thm]{Corollary}
\newtheorem{prop}[thm]{Proposition}
\newtheorem{lemma}[thm]{Lemma}

\newtheorem*{claim*}{Claim}

\theoremstyle{definition}

\theoremstyle{remark}
\newtheorem*{remark}{Remark}

\title{Zeros of Theta Functions Associated with Self-Dual Lattices}
\author{Roei Raveh}
\address{School of Mathematical Sciences\\
Tel Aviv University\\
Tel Aviv 69978, Israel}
\email{roeiraveh@mail.tau.ac.il}
\date{}

\begin{document}
\begin{abstract}
    We study the zeros of theta functions  $\Theta_{\Gamma_{4k}}$ associated with the lattices $\Gamma_{4k}$, a family of self-dual lattices generalizing the $\mathsf{E}_{8}$ lattice. Our results show two different behaviors of the zeros according to the lattice parity:
    When $\Gamma_{4k}$ is an even lattice, we show that the zeros all lie on the line $\Re z =\frac{1}{2}$ in the fundamental domain and prove that the zeros are equidistributed with respect to an explicit probability measure on the line $\Re z = \frac{1}{2}$. However, when the $\Gamma_{4k}$ is an odd lattice, there are no zeros on the line $\Re z =\frac{1}{2}$, only exponentially close to it.
    Our argument relies on representing $\Theta_{\Gamma_{4k}}$ as a polynomial in the modular $\lambda$-function. We then study the zeros of this polynomial and exploit some conformal properties of $\lambda$ to get our results.
\end{abstract}
\maketitle

\section{Introduction} 


\subsection{Zeros of modular forms}
The zeros of modular forms are an important part of the theory of modular forms; by the work of Bruinier, Kohnen, and Ono, the zeros are directly related to the Fourier coefficients of the form and can relate to values of certain $L$-series, \cite[see Theorem 3]{OnoKohnenBruinier}.

For any $f$ a nonzero modular form of weight $k$, we have the valence formula (see \cite{serre}):
\begin{equation}\label{eq: valence}
    \ord_{\infty}\pr{f} + \frac{1}{2}\ord_{i}\pr{f} + \frac{1}{3}\ord_{\rho}\pr{f} + \sum_{z\in\cF\setminus\{i,\rho\}}\! \ord_{z}\pr{f} =\frac{k}{12},
\end{equation}
where $\rho = e^{{i\pi}/{3}} = \frac{1}{2}+i\frac{\sqrt{3}}{2}$ and 
\[\cF = \prs{z\in\HH:\abs{z}> 1, -\frac{1}{2}< \Re z< 0}\bigcup \prs{z\in\HH:\abs{z}\ge 1, 0\le \Re z\le \frac{1}{2}},\]
as demonstrated Figure \ref{fig: fundomain}.
\begin{figure}[ht]
    \centering
    \begin{tikzpicture}[scale = 2] 
    \draw[thick,->] (-1.25,0) -- (1.25,0) node[right] {};
    \draw[fill=gray!30,draw =white] (-0.5,2.6) -- (120:1) arc (120:60:1) -- (0.5,2.6);
    \draw[line width=0.2mm] (60:1) arc (60:90:1);
    \draw[line width=0.15mm,dashed] (0,1) arc (90:120:1);
    \draw[line width=0.15mm,dashed] (120:1) -- (-0.5,2.6);
    \draw[line width=0.2mm] (60:1) -- (0.5, 2.6);
    \draw[line width=0.15mm,loosely dashed] (1,0) arc (0:60:1);
    \draw[line width=0.15mm,loosely dashed](120:1) arc (120:180:1);
    \draw[loosely dashed] (0.5,0) -- (60:1);
    \draw[loosely dashed] (-0.5,0) -- (120:1);
    \filldraw (0,0) circle (0.75pt) node[below] {$0$};
    \filldraw (1,0) circle (0.75pt) node[below right] {$1$};
    \filldraw (-1,0) circle (0.75pt) node[below left] {$-1$};
    \filldraw (-0.5,0) circle (0.75pt) node[below] {$-\frac{1}{2}$};
    \filldraw (0.5,0) circle (0.75pt) node[below] {$\frac{1}{2}$};
    \filldraw (0,1) circle (0.75pt) node[above] {$i$};
    \filldraw (60:1) circle (0.75pt) node[above right] {$\rho$};
    \filldraw[gray!30] (0,1.8) circle (0.75pt) node {\color{black}$\cF$};
    \end{tikzpicture}
    \caption{The fundamental domain $\cF$.}
    \label{fig: fundomain}
\end{figure}
The valence formula \eqref{eq: valence} implies that for a nonzero modular form of weight $k$, there are about ${\frac{k}{12}}$ zeros in the fundamental domain $\cF$, which raises a natural question:
\begin{quote}
    For a distinguished family of modular forms, can we find the location or limit distribution of their zeros in the fundamental domain? 
\end{quote}
The first result regarding this inquiry was given in 1970 by F. Rankin and Swinnerton-Dyer, \cite{sdr}. They proved that the zeros of the Eisenstein series in the fundamental domain lie on the arc $\cA=\left\{e^{i\varphi}:\frac{\pi}{3}\le \varphi\le \frac{\pi}{2}\right\}$ and become uniformly distributed in $\cA$ as $k\to\infty$. 
For different types of results on the zeros of various modular forms, see \cite{choi, dukejen, gun, holowinsky2010mass, Kimmel, rankin, raveh, RVY, rudnick2005, rudnick2023, Cornelissen,zilka}.

This paper concerns the zeros of the theta function of the lattices $\Gamma_{4k}$, a family of self-dual lattices.
\subsection{Theta functions and self-dual lattices}
For any lattice $\Lambda\sub \Rd$, define the corresponding theta function $\Theta_{\Lambda}:\HH\to\C$ as
\begin{equation}\label{eq: theta fn def}
    \Theta_{\Lambda}\pr{\tau} = \sum_{x\in\Lambda}e^{\pi i\tau\norm{x}^{2}}.
\end{equation}
Theta functions play a crucial role in the study of integral quadratic forms and are closely tied to modular forms.
If the lattice is integral, i.e., $\norm{x}^2 \in \Z$ for all $x\in \Lambda$, then $\Theta_{\Lambda}$ is a modular form of weight $d/2$ of some level $N$. A particularly interesting case is when $\Lambda$ is self-dual and  $\Lambda=\Lambda'$, then the function $\Theta_{\Lambda}$ is a modular form of weight $d/2$ of level $2$. Integral lattices are assigned a parity: lattices for which $\norm{x}^{2}\in 2\Z$ for all $x\in \Lambda$ are called \emph{even}, and non-even lattices are called \emph{odd}. Even self-dual lattices are called even unimodular lattices, and any theta function associated with an even unimodular lattice is a modular form for $\slz$. 

Even unimodular lattices have been a subject of interest among mathematicians in the last century. For a given $d\ge 1$, there are finitely many even unimodular lattices in $\R^{d}$. 
In dimension $8$, there exists a unique (up to equivalences) even unimodular lattice called the $\mathsf{E}_{8}$ lattice (see \cite{ConwaySloane}). Smith \cite{Smith} proved its existence in 1867, and Korkine and Zolotarev \cite{KorkineZolotareff} followed with an explicit construction in 1873. The $\mathsf{E}_{8}$ lattice (and its higher-dimensional generalizations) is at the center of many applications in mathematics. For example, it is in the core of Viazovska's work on the sphere packing problem in dimension $8$ \cite{SpherePackingViazovska}, it appears in the work of Freedman on exotic $4$-manifolds \cite{Freedman}, and was used to construct a counter-example for Mark Kac's \say{Can one hear the shape of a drum?} \cite{123, HearingDrums}.

The lattices $\Gamma_{n}$ are the higher-dimensional generalizations of the famous $\mathsf{E}_{8}$ lattice.\footnote{These are actually generalizations of what is known as the even coordinate system of $\mathsf{E}_{8}$. They have remarkable packing properties; see \cite[p. 119-120]{ConwaySloane}.} They are constructed as follows: For any $n\ge 1$, define
\begin{equation}\label{eq: lattice def}
    \Gamma_{n}=D_{n}\cup\pr{\delta_{n}+D_{n}}.
\end{equation}
where 
\begin{equation}\label{eq: Dn lattice def}
    D_{n} = \prs{\pr{x_1, \ldots,x_{n}}\in \Z^{n}:\sum_{i=1}^{n}x_i \equiv 0\pmod{2}},\text{ and } \delta_{n} =\pr{\frac{1}{2},\ldots,\frac{1}{2}}\in\R^{n}.
\end{equation} 
For any $n\ge 1$, the lattice $\Gamma_{n}$ is self-dual if and only if $4\mid n$ and even if and only if $8\mid n$ (see \cite{serre}). In this paper, we provide a comprehensive understanding of the location and limiting distribution of the zeros of $\Theta_{\Gamma_{4k}}$.
As $\Theta_{\Gamma_{4k}}$ are all modular forms of level $2$, we can study their zeros in the fundamental domain of $\Gamma\pr{2}$, the principal subgroup of level $2$. In the fundamental domain of $\Gamma\pr{2}$, the zeros of $\Theta_{\Gamma_{4k}}$ exhibit a strong pattern.
Informally, they are \say{attracted} to the six geodesics $\cL_{\rho}$, $\cL_{\rho}^*$, $\cC$, $\cC^{*}$, $\cU$, and $\cU^*$ (see Figure \ref{fig: theta zeros pattern}), where
\begin{align}
    \label{eq: line 0.5} \cL_{\rho} &  = \prs{\frac{1}{2}+it:t>\frac{\sqrt{3}}{2}},\\
    \label{eq: the arc center 1} \cC&  = \prs{1+e^{i\varphi}:\varphi\in\left(\frac{2\pi}{3},\pi\right)},\\
     \label{eq: unit arc} \cU & = {\prs{e^{i\varphi}:\varphi\in\left(0,\frac{\pi}{3}\right)}},
\end{align}
and  $\cL_{\rho}^*$, $\cC^*$, and $\cU^*$ denote their respective reflections along the imaginary axis.

\begin{figure}[ht]
    \centering
    \begin{subfigure}[b]{0.3\textwidth}
    \begin{tikzpicture}[scale = 1.7] 
        \definecolor{newblue}{RGB}{26, 128, 187};
        \draw[gray,thick,->] (-1.1,0) -- (1.1,0) node[right] {};
        \draw[fill=gray!10,draw =white] (-1,2) -- (-1,0) arc (180:0:0.5) arc (180:0:0.5) -- (1,2);
        \draw[gray, line width=0.2mm] (-1,0) -- (-1,2);
        \draw[gray, line width=0.15mm,dashed] (1,0) -- (1,2);
        \draw[gray, line width=0.2mm] (0,0) arc (0:180:0.5);
        \draw[gray, line width=0.15mm,dashed] (1,0) arc (0:180:0.5);
        \draw[gray, line width = 0.2mm] (60:1) -- (0.5,2);
        \draw[gray, line width = 0.2mm] (120:1) -- (-0.5,2);
        \draw[line width=0.2mm, gray] (1,0) arc (0:60:1);
        \draw[line width=0.2mm, gray] (0,0) arc (180:120:1);
        \draw[line width=0.2mm, gray] (0,0) arc (0:60:1);
        \draw[line width=0.2mm, gray] (-1,0) arc (180:120:1);
        \filldraw[gray] (1,0) circle (0.75pt) node[below] {\tiny$1$};
        \filldraw[gray] (-1,0) circle (0.75pt) node[below] {\tiny$-1$};
        \filldraw[gray] (0,0) circle (0.75pt) node[below] {\tiny$0$};
        \filldraw[gray] (-0.5,0) circle (0.75pt) node[below] {\tiny$-\frac{1}{2}$};
        \filldraw[gray] (0.5,0) circle (0.75pt) node[below] {\tiny$\frac{1}{2}$};  
        \filldraw[gray] (120:1) circle (0.75pt) node[above left] {\tiny$-\bar{\rho}$};
        \filldraw[gray] (60:1) circle (0.75pt) node[above right] {\tiny$\rho$};
        \foreach \point in {(0.5, 0.866025), (-0.5, 0.866025), (-0.793036, 0.609174), (0.793036,  0.609174), (-0.668465, 0.743744), (0.668465, 0.743744), (-0.572048,  0.82022), (0.572048, 0.82022), (0.5, 0.958308), (-0.427952, 0.82022),  (-0.5, 0.958308), (0.427952, 0.82022), (0.5, 1.12167), (-0.331535,  0.743744), (-0.5, 1.12167), (0.331535, 0.743744), (0.5, 1.47169),  (-0.206964, 0.609174), (-0.5, 1.47169), (0.206964, 0.609174)}{
            \filldraw[newblue] \point circle (0.75pt);
        }
    \end{tikzpicture}
    \end{subfigure}
    \begin{subfigure}[b]{0.3\textwidth}
    \begin{tikzpicture}[scale = 1.7] 
        \definecolor{newblue}{RGB}{26, 128, 187};
        \draw[gray,thick,->] (-1.1,0) -- (1.1,0) node[right] {};
        \draw[fill=gray!10,draw =white] (-1,2) -- (-1,0) arc (180:0:0.5) arc (180:0:0.5) -- (1,2);
        \draw[gray, line width=0.2mm] (-1,0) -- (-1,2);
        \draw[gray, line width=0.15mm,dashed] (1,0) -- (1,2);
        \draw[gray, line width=0.2mm] (0,0) arc (0:180:0.5);
        \draw[gray, line width=0.15mm,dashed] (1,0) arc (0:180:0.5);
        \draw[gray, line width = 0.2mm] (60:1) -- (0.5,2);
        \draw[gray, line width = 0.2mm] (120:1) -- (-0.5,2);
        \draw[line width=0.2mm, gray] (1,0) arc (0:60:1);
        \draw[line width=0.2mm, gray] (0,0) arc (180:120:1);
        \draw[line width=0.2mm, gray] (0,0) arc (0:60:1);
        \draw[line width=0.2mm, gray] (-1,0) arc (180:120:1);
        \filldraw[gray] (1,0) circle (0.75pt) node[below] {\tiny$1$};
        \filldraw[gray] (-1,0) circle (0.75pt) node[below] {\tiny$-1$};
        \filldraw[gray] (0,0) circle (0.75pt) node[below] {\tiny$0$};
        \filldraw[gray] (-0.5,0) circle (0.75pt) node[below] {\tiny$-\frac{1}{2}$};
        \filldraw[gray] (0.5,0) circle (0.75pt) node[below] {\tiny$\frac{1}{2}$};  
        \filldraw[gray] (120:1) circle (0.75pt) node[above left] {\tiny$-\bar{\rho}$};
        \filldraw[gray] (60:1) circle (0.75pt) node[above right] {\tiny$\rho$};
        \foreach \point in {(-0.781902, 0.623401), (0.781902, 0.623401), (-0.707634, 0.706579),(0.707634, 0.706579), (-0.64648, 0.762931), (0.64648, 0.762931),(-0.591626, 0.806213), (0.591626, 0.806213), (-0.54062, 0.841267),(0.54062, 0.841267), (0.501382, 0.885149), (-0.484486, 0.855322),(-0.501382, 0.885149), (0.484486, 0.855322), (0.499998, 0.949379),(-0.434283, 0.824602), (-0.499998, 0.949379), (0.434283, 0.824602),(0.5, 1.02981), (-0.381533, 0.785811), (-0.5, 1.02981), (0.381533,0.785811), (0.5, 1.13723), (-0.323982, 0.736885), (-0.5, 1.13723),(0.323982, 0.736885), (0.5, 1.30005), (-0.257716, 0.670085), (-0.5,1.30005), (0.257716, 0.670085), (0.5, 1.64985), (-0.168236, 0.55513),(-0.5, 1.64985), (0.168236, 0.55513)}{
            \filldraw[newblue] \point circle (0.75pt);
        }
    \end{tikzpicture}
    \end{subfigure}
    \begin{subfigure}[b]{0.3\textwidth}
    \begin{tikzpicture}[scale = 1.7] 
        \definecolor{newblue}{RGB}{26, 128, 187};
        \draw[gray,thick,->] (-1.1,0) -- (1.1,0) node[right] {};
        \draw[fill=gray!10,draw =white] (-1,2) -- (-1,0) arc (180:0:0.5) arc (180:0:0.5) -- (1,2);
        \draw[gray, line width=0.2mm] (-1,0) -- (-1,2);
        \draw[gray, line width=0.15mm,dashed] (1,0) -- (1,2);
        \draw[gray, line width=0.2mm] (0,0) arc (0:180:0.5);
        \draw[gray, line width=0.15mm,dashed] (1,0) arc (0:180:0.5);
        \draw[gray, line width = 0.2mm] (60:1) -- (0.5,2);
        \draw[gray, line width = 0.2mm] (120:1) -- (-0.5,2);
        \draw[line width=0.2mm, gray] (1,0) arc (0:60:1);
        \draw[line width=0.2mm, gray] (0,0) arc (180:120:1);
        \draw[line width=0.2mm, gray] (0,0) arc (0:60:1);
        \draw[line width=0.2mm, gray] (-1,0) arc (180:120:1);
        \filldraw[gray] (1,0) circle (0.75pt) node[below] {\tiny$1$};
        \filldraw[gray] (-1,0) circle (0.75pt) node[below] {\tiny$-1$};
        \filldraw[gray] (0,0) circle (0.75pt) node[below] {\tiny$0$};
        \filldraw[gray] (-0.5,0) circle (0.75pt) node[below] {\tiny$-\frac{1}{2}$};
        \filldraw[gray] (0.5,0) circle (0.75pt) node[below] {\tiny$\frac{1}{2}$};  
        \filldraw[gray] (120:1) circle (0.75pt) node[above left] {\tiny$-\bar{\rho}$};
        \filldraw[gray] (60:1) circle (0.75pt) node[above right] {\tiny$\rho$};
        \foreach \point in {(0.5, 0.866025), (-0.5, 0.866025), (-0.851171, 0.524889), (0.851171,0.524889), (-0.777618, 0.628737), (0.777618, 0.628737), (-0.724494,0.689282), (0.724494, 0.689282), (-0.679074, 0.734069), (0.679074,0.734069), (-0.637935, 0.77009), (0.637935, 0.77009), (-0.599571,0.800322), (0.599571, 0.800322), (-0.563169, 0.826342), (0.563169,0.826342), (-0.52819, 0.849127), (0.52819, 0.849127), (0.5, 0.89986),(-0.47181, 0.849127), (-0.5, 0.89986), (0.47181, 0.849127), (0.5,0.945837), (-0.436831, 0.826342), (-0.5, 0.945837), (0.436831,0.826342), (0.5, 0.99933), (-0.400429, 0.800322), (-0.5, 0.99933),(0.400429, 0.800322), (0.5, 1.06347), (-0.362065, 0.77009), (-0.5,1.06347), (0.362065, 0.77009), (0.5, 1.14368), (-0.320926, 0.734069),(-0.5, 1.14368), (0.320926, 0.734069), (0.5, 1.25094), (-0.275506,0.689282), (-0.5, 1.25094), (0.275506, 0.689282), (0.5, 1.41364),(-0.222382, 0.628737), (-0.5, 1.41364), (0.222382, 0.628737), (0.5,1.76339), (-0.148829, 0.524889), (-0.5, 1.76339), (0.148829, 0.524889)}{
            \filldraw[newblue] \point circle (0.75pt);
        }
    \end{tikzpicture}
    \end{subfigure}
    \caption{The zeros of $\Theta_{\Gamma_{4k}}$ where $k\in\prs{20,35,60}$.}
    \label{fig: theta zeros pattern}
\end{figure}
\subsection{Statements and results}
We divide our results into results on the even case, i.e., when $\Gamma_{n}$ is an even unimodular lattice, and results on the odd case, i.e., when $\Gamma_{n}$ is self-dual and odd. We will begin with the even case, which is slightly easier. 
\subsubsection*{The even case}  
In this case we consider the functions $\Theta_{\Gamma_{8k}}$ for $k\ge 1$. Since the lattice $\Gamma_{8k}$ is even, the function $\Theta_{\Gamma_{8k}}$ is a modular form of weight $4k$ for $\slz$. Therefore, we will study its zeros in the fundamental domain $\cF$. 
Throughout this case, we write $4k=12\ell+k'$ with $\ell=\floor{\frac{k}{3}}$ and $k'\in\prs{0,4,8}$. 
Our main results are as follows:
\begin{thm}\label{thm: Zeros}
    For all $k\ge1$, the zeros of $\Theta_{\Gamma_{8k}}$ are all simple (except for $\rho$) and lie on the line $\cL_{\rho}$.
    Furthermore, let $\tau_{k,1},\ldots,\tau_{k,\ell}\in \cL_{\rho}$ be the zeros of $\Theta_{\Gamma_{8k}}$ ordered with decreasing imaginary value,  i.e. $\Im\tau_{k,\ell}<\ldots<\Im\tau_{k,1}$. Then the highest zeros $\tau_{1},\tau_{2},\ldots$ satisfy
    \[\Im\tau_{k,m} =\frac{1}{\pi}\log(\frac{16k}{\pi m}) + o\pr{1},\]
    as $k\to\infty$ and $m=o\pr{k}$.
\end{thm}

\begin{figure}[ht]
    \centering
    \begin{tikzpicture}[scale = 3] 
    \definecolor{newblue}{RGB}{26, 128, 187};
    \draw[thick,->] (-1.25,0) -- (1.25,0) node[right] {};
    \draw[fill=gray!10,draw =white] (-0.5,2.2) -- (120:1) arc (120:60:1) -- (0.5,2.2);
    \draw[line width=0.2mm,gray] (60:1) arc (60:90:1);
    \draw[line width=0.15mm,dashed,gray] (0,1) arc (90:120:1);
    \draw[line width=0.15mm,dashed,gray] (120:1) -- (-0.5,2.2);
    \draw[line width=0.2mm,gray] (60:1) -- (0.5, 2.2);
    \draw[line width=0.15mm,loosely dashed,gray] (1,0) arc (0:60:1);
    \draw[line width=0.15mm,loosely dashed,gray](120:1) arc (120:180:1);
    \draw[loosely dashed,gray] (0.5,0) -- (60:1);
    \draw[loosely dashed,gray] (-0.5,0) -- (120:1);
    \filldraw (0,0) circle (0.5pt) node[below] {$0$};
    \filldraw (1,0) circle (0.5pt) node[below right] {$1$};
    \filldraw (-1,0) circle (0.5pt) node[below left] {$-1$};
    \filldraw (-0.5,0) circle (0.5pt) node[below] {$-\frac{1}{2}$};
    \filldraw (0.5,0) circle (0.5pt) node[below] {$\frac{1}{2}$};
    \filldraw[gray] (0,1) circle (0.5pt) node[below] {$i$};
    \filldraw[gray] (60:1) circle (0.5pt) node[below left] {$\rho$};
    \filldraw[newblue] (0.5,2.04207) circle (0.5pt) node[right] {$\tau_{1}$};
    \filldraw[newblue] (0.5,1.69236) circle (0.5pt) node[right] {$\tau_{2}$};
    \filldraw[newblue] (0.5,1.52974) circle (0.5pt) node[right] {$\tau_{3}$};
    \filldraw[newblue] (0.5,0.874333) circle (0.5pt) node[right] {$\tau_{20}$};
    \foreach \y in {0.891108, 0.908964, 0.92785, 0.947899, 0.969265, 0.99214, 1.01676, 1.04341, 1.07247, 1.10441, 1.1399, 1.17981, 1.22543, 1.27866, 1.34258, 1.42261}{
        \filldraw[newblue] (0.5,\y) circle (0.5pt);
    }
    \end{tikzpicture}
    \caption{The zeros of $\Theta_{\Gamma_{480}}$ in $\cF$.}
    \label{fig: zeros of theta}
\end{figure}
We also study the distribution of the zeros on $\cL_{\rho}$ and show that they are equidistributed on $\cL_{\rho}$ with respect to some density, i.e., an absolutely continuous measure with respect to the $1$-dimensional Lebesgue measure on $\cL_{\rho}$:
\begin{thm}\label{thm: Zeros Distrbution}
    Let $\lambda$ be the modular lambda function and $\varrho\pr{y}=\frac{3}{\pi}\frac{\lambda'\pr{\frac{1}{2}+iy}}{\lambda\pr{\frac{1}{2}+iy}-1}$. Then the zeros of $\Theta_{\Gamma_{8k}}$ are equidistributed on $\cL_{\rho}$ with respect to the measure $\varrho\pr{y}dy$. In addition, we have
    \[\varrho\pr{y}\sim 48e^{-\pi y},\quad \text{as }y\to\infty.\]
\end{thm}

\subsubsection*{The odd case} 
In this case we consider the functions $\Theta_{\Gamma_{8k}}$ for $k\ge 1$. Since the lattice $\Gamma_{8k+4}$ is self-dual, the function $\Theta_{\Gamma_{8k+4}}$ is a modular form of weight $4k+2$ of level $2$. Therefore, we will study its zeros in the fundamental domain $\cF_{\lambda}$. 
In this case, for all $k\ge 1$ we write $4k+2=12\ell+k'$ with $k'\in\prs{6,10,14}$.
In this case, we have the following:
\begin{thm}\label{thm: Zeros odd case}
    For all $k\ge1$, the zeros of $\Theta_{\Gamma_{8k+4}}$ are all simple. Furthermore, there are at least $\ell$ zeros on each of the geodesics $\cU$ and $\cU^*$, a simple zero at the cusp $\mathfrak{1}$, and no zeros on the geodesics  $\cL_{\rho}$, $\cL_{\rho}^*$, $\cC$, and $\cC^{*}$.
\end{thm}
Theorem \ref{thm: Zeros odd case} is in striking difference to Theorem \ref{thm: Zeros}, as the latter implies that there are $\ell$ simple zeros on each  of the geodesic $\cL_{\rho}$, $\cL_{\rho}^*$, $\cC$, $\cC^{*}$, $\cU$, and $\cU^*$ since those are all equivalent under the action of $\slz$. However, the zeros are \say{attracted} to the geodesics $\cL_{\rho}$, $\cL_{\rho}^*$, $\cC$, and $\cC^{*}$ in an exponential rate:
\begin{thm}\label{thm: zeros odd case decay}
    For any $\alpha \in \pr{0,\frac{1}{3}}$ and $k\gg 1$, there exist at least $m\ge k\alpha-2$ distinct zeros $\tau_{k,1},\ldots,\tau_{k,m}\in \cF_{\lambda}$ of $\Theta_{\Gamma_{8k+4}}$ such that
    \[\abs{\Re\pr{\tau_{k,j}}-\frac{1}{2}} \ll_{a}{k^{-1}e^{-c_{\alpha}k}}.\]
\end{thm}
In fact, one can choose $c_{\alpha} = -2\log\pr{\frac{1+2\cos\pr{\frac{5\pi-3\alpha\pi}{12}}}{2}}$.
\begin{figure}[ht]
    \centering
    \begin{tikzpicture}[scale = 3] 
    \definecolor{newblue}{RGB}{26, 128, 187};
    \draw[gray,thick,->] (-1.1,0) -- (1.1,0) node[right] {};
    \draw[fill=gray!10,draw =white] (-1,2.2) -- (-1,0) arc (180:0:0.5) arc (180:0:0.5) -- (1,2.2);
    \draw[gray, line width=0.2mm] (-1,0) -- (-1,2.2);
    \draw[gray, line width=0.15mm,dashed] (1,0) -- (1,2.2);
    \draw[gray, line width=0.2mm] (0,0) arc (0:180:0.5);
    \draw[gray, line width=0.15mm,dashed] (1,0) arc (0:180:0.5);
    \draw[gray, line width = 0.2mm] (60:1) -- (0.5,2.2);
    \draw[gray, line width = 0.2mm] (120:1) -- (-0.5,2.2);
    \draw[line width=0.2mm, gray] (1,0) arc (0:60:1);
    \draw[line width=0.2mm, gray] (0,0) arc (180:120:1);
    \draw[line width=0.2mm, gray] (0,0) arc (0:60:1);
    \draw[line width=0.2mm, gray] (-1,0) arc (180:120:1);   
    \filldraw (0,0) circle (0.75pt) node[below] {$0$};
    \filldraw (1,0) circle (0.75pt) node[below right] {$1$};
    \filldraw (-1,0) circle (0.75pt) node[below left] {$-1$};
    \filldraw (-0.5,0) circle (0.75pt) node[below] {$-\frac{1}{2}$};
    \filldraw (0.5,0) circle (0.75pt) node[below] {$\frac{1}{2}$};
    \filldraw[gray] (120:1) circle (0.5pt) node[above left] {$-\bar{\rho}$};
    \filldraw[gray] (60:1) circle (0.5pt) node[above right] {$\rho$};
    \foreach \y in {(-0.860227, 0.509911), (0.860227, 0.509911), (-0.822756, 0.568395),(0.822756, 0.568395), (-0.793701, 0.608308), (0.793701, 0.608308),(-0.768727, 0.639577), (0.768727, 0.639577), (-0.746252, 0.665663),(0.746252, 0.665663), (-0.725495, 0.688227), (0.725495, 0.688227),(-0.706005, 0.708207), (0.706005, 0.708207), (-0.687494, 0.72619),(0.687494, 0.72619), (-0.669767, 0.742571), (0.669767, 0.742571),(-0.652683, 0.757631), (0.652683, 0.757631), (-0.636139, 0.771575),(0.636139, 0.771575), (-0.620053, 0.78456), (0.620053, 0.78456),(-0.604363, 0.796709), (0.604363, 0.796709), (-0.589018, 0.80812),(0.589018, 0.80812), (-0.573977, 0.818871), (0.573977, 0.818871),(-0.559206, 0.829029), (0.559206, 0.829029), (-0.544676, 0.838646),(0.544676, 0.838646), (-0.530363, 0.847771), (0.530363, 0.847771),(-0.516243, 0.856442), (0.516243, 0.856442), (-0.502987, 0.864294),(0.502987, 0.864294), (0.499935, 0.876848), (-0.490711, 0.860669),(-0.499935, 0.876848), (0.490711, 0.860669), (0.5, 0.893777),(-0.476719, 0.85216), (-0.5, 0.893777), (0.476719, 0.85216), (0.5,0.911629), (-0.462506, 0.843268), (-0.5, 0.911629), (0.462506,0.843268), (0.5, 0.930513), (-0.448087, 0.833902), (-0.5, 0.930513),(0.448087, 0.833902), (0.5, 0.950558), (-0.43344, 0.824021), (-0.5,0.950558), (0.43344, 0.824021), (0.5, 0.971923), (-0.418538,0.813574), (-0.5, 0.971923), (0.418538, 0.813574), (0.5, 0.994795),(-0.40335, 0.802502), (-0.5, 0.994795), (0.40335, 0.802502), (0.5,1.01941), (-0.387839, 0.790733), (-0.5, 1.01941), (0.387839,0.790733), (0.5, 1.04606), (-0.371957, 0.778179), (-0.5, 1.04606),(0.371957, 0.778179), (0.5, 1.07512), (-0.355651, 0.764732), (-0.5,1.07512), (0.355651, 0.764732), (0.5, 1.10706), (-0.338848,0.750252), (-0.5, 1.10706), (0.338848, 0.750252), (0.5, 1.14255),(-0.321458, 0.734561), (-0.5, 1.14255), (0.321458, 0.734561), (0.5,1.18246), (-0.303359, 0.71742), (-0.5, 1.18246), (0.303359, 0.71742),(0.5, 1.22807), (-0.284388, 0.698498), (-0.5, 1.22807), (0.284388,0.698498), (0.5, 1.2813), (-0.264308, 0.677316), (-0.5, 1.2813),(0.264308, 0.677316), (0.5, 1.34522), (-0.242763, 0.65314), (-0.5,1.34522), (0.242763, 0.65314), (0.5, 1.42525), (-0.219169, 0.624742),(-0.5, 1.42525), (0.219169, 0.624742), (0.5, 1.53238), (-0.192441,0.589787), (-0.5, 1.53238), (0.192441, 0.589787), (0.5, 1.695),(-0.160101, 0.542743), (-0.5, 1.695), (0.160101, 0.542743), (0.5,2.04471), (-0.112845, 0.461472), (-0.5, 2.04471), (0.112845, 0.461472)}{
        \filldraw[newblue] \y circle (0.5pt);
    }
    \end{tikzpicture}
    \caption{The zeros of $\Theta_{\Gamma_{484}}$ in $\cF_{\lambda}$.}
    \label{fig: odd zeros of theta}
\end{figure}

\begin{remark}
By symmetry, the same can be stated for $\cL_{\rho}^*$, and one can state a result of this nature for $\cC$ and $\cC^*$. Furthermore, using our arguments, one can easily deduce that the zeros on $\cU$ and $\cU^*$ satisfy some equidistribution result as in Theorem \ref{thm: Zeros Distrbution}.
\end{remark}

Our argument uses Jacobi theta functions and the modular lambda function. In \S\ref{sec: background}, we provide background on lattices, theta functions, and modular forms. In \S\ref{sec: lambda conf pro} we prove some conformal properties of the modular lambda functions and study the structure of $\Theta_{\Gamma_{4k}}$. In \S\ref{sec: zeros} we prove Theorem \ref{thm: Zeros} and  \S\ref{sec: distribution} we prove Theorem \ref{thm: Zeros Distrbution} and in \S\ref{sec: odd case} we deal with the odd case and prove Theorem \ref{thm: Zeros odd case} and Theorem \ref{thm: zeros odd case decay}.
\subsection*{Acknowledgments}
We are grateful to Ze{\'e}v Rudnick, Misha Sodin, and Adi Zilka for helpful discussions. We also extend our gratitude to SoYoung Choi for her comments on an earlier version of this manuscript.
This research was supported by the ISRAEL SCIENCE FOUNDATION (No. 2860/24).

\section{Background}\label{sec: background}
\subsection{Modular forms on $\slz$}
A modular form of weight $k$ for $\slz$ is a holomorphic function $f:\HH \to \C$ satisfying  
\begin{equation}\label{eq: modular def 1}
    f{\left(\frac{a\tau+b}{c\tau +d}\right)}={\left(c\tau+d\right)^{k}}f(\tau),\quad\forall \begin{psmallmatrix}
    a & b \\
    c & d
    \end{psmallmatrix}\in\slz,
\end{equation}
and remains bounded as $\Im\tau\to\infty$. If $f$ vanishes as $\Im\tau\to\infty$, it is called a cusp form. Due to the symmetric nature of the transformation formula \eqref{eq: modular def 1}, modular forms can be viewed as functions defined on the fundamental domain $\cF$, as seen in Figure \ref{fig: fundomain}.

\begin{remark}
We can replace \eqref{eq: modular def 1} with the following conditions:
\begin{align}
    f(\tau) & = f(\tau+1),\label{eq:1periodic}\\
    f(\tau) & = \tau^{-k}f(-1/\tau). \label{eq:inverting}
\end{align}
\end{remark}
When $k\ge 4$ and even, there exists a nonzero modular form in $M_{k}$ known as the (normalized) Eisenstein series
\begin{equation}\label{eq: esn}
    E_{k}(\tau) = \frac{1}{2}\msum{(m,n)\in\Z^{2}}{\gcd(m,n)=1}\frac{1}{\left(m\tau+n\right)^{k}}= 1 -\gamma_{k}\sum_{n=1}^{\infty}\sigma_{k-1}(n)q^{n},
\end{equation}
where $\sigma_{k-1}(n)=\sum_{d\mid n}d^{k-1}$, $\gamma_{k}=\frac{2k}{B_k}$, and $B_k$ is the $k$-th Bernoulli number.

\subsection{Modular forms on $\Gamma\pr{2}$} Let $\Gamma\pr{2}$ be the principal congruence subgroup of level $2$, i.e., 
\[\Gamma\pr{2} = \prs{\begin{pmatrix}
    a & b\\
    c &d
\end{pmatrix} \in \slz :\begin{pmatrix}
    a & b\\
    c &d
\end{pmatrix} \equiv \begin{pmatrix}
    1 & 0\\
    0 &1
\end{pmatrix}\pmod{2}}.\]
A modular form of weight $k$ for $\Gamma\pr{2}$ is a holomorphic function $f:\HH \to \C$ satisfying  
\begin{equation}\label{eq: lvl2 modular def 1}
    f{\left(\frac{a\tau+b}{c\tau +d}\right)}={\left(c\tau+d\right)^{k}}f(\tau),\quad\forall \begin{psmallmatrix}
    a & b \\
    c & d
    \end{psmallmatrix}\in\Gamma\pr{2},
\end{equation}
and the function $\pr{c\tau+d}^{-k}f{\left(\frac{a\tau+b}{c\tau +d}\right)}$ remains bounded as $\Im\tau\to\infty$, for all $\begin{psmallmatrix}
    a & b \\
    c & d
    \end{psmallmatrix}\in\slz$.
\begin{remark}
We can replace \eqref{eq: lvl2 modular def 1} with the following conditions:
\begin{align}
    f(\tau) & = f(\tau+2),\label{eq: lvl2 2periodic}\\
    f\pr{\frac{\tau}{2\tau+1}}& = \pr{2\tau+1}^{k}f\pr{\tau}. \label{eq: lvl2 inverting}
\end{align}
\end{remark}
We will denote $M_{k}\pr{2}$ for the space of modular forms for $\Gamma\pr{2}$ of weight $k$. 
As $2$-periodic functions, each $f\in M_{k}\pr{2}$, a modular form for $\Gamma\pr{2}$, has a Fourier series, given in terms of the regular nome $q= e^{2\pi i\tau}$:
\[f(\tau) = \sum_{n=n_{\infty}}^{\infty}a_{f}\pr{n}q^{n/2},\]
or in terms of the normalized nome:
\[f(\tau) = \sum_{n=n_{\infty}}^{\infty}a_{f}\pr{n}q_{2}^{n}, \quad q_{2}= e^{\pi i\tau}.\]
The subgroup $\Gamma\pr{2}$ has $3$ cusps: $\mathfrak{0}$, $\mathfrak{1}$, and $\mathfrak{\infty}$. If $\mathfrak{a}$ is a cusp and $\sigma_{\mathfrak{a}}\infty =\mathfrak{a}$ where $\sigma_{\mathfrak{a}}\in\slz$, for any $f\in M_{k}\pr{2}$ we can define the Fourier series at the cusp $\mathfrak{a}$,
\[f\pr{\sigma_{\mathfrak{a}}\tau} = \sum_{n=n_{\mathfrak{a}}}^{\infty}a_{f,\mathfrak{a}}\pr{n}q^{n/2}.\] The growth condition on $f$ as $\Im\tau\to\infty$ implies that $n_{\mathfrak{a}}$ is an integer (see \cite{Iwaniec}). We denote $\ord_{\mathfrak{a}}\pr{f}=n_{\mathfrak{a}}$ and state an analogous valence formula for $\Gamma\pr{2}$, for a given nonzero $f\in M_{k}\pr{2}$:
\begin{equation}\label{eq: lvl2 valence}
    \ord_{\infty}\pr{f}+\ord_\mathfrak{1}\pr{f}+\ord_\mathfrak{0}\pr{f}+\sum_{z\in\cF_{\lambda}}\ord_{z}\pr{f} = \frac{k}{2}.
\end{equation}
where $\cF_{\lambda}$ is the fundamental domain \[\cF_{\lambda}= \prs{\tau\in\HH: \abs{\tau+\frac{1}{2}}\ge \frac{1}{2},-1\le \Re\tau\le 0}\bigcup\prs{\tau\in\HH: \abs{\tau-\frac{1}{2}}> \frac{1}{2},0< \Re\tau< 1},\]
as demonstrated Figure \ref{fig: fundomain 2}, also see \cite{DiamondShurman}.

\begin{figure}[ht]
    \centering
    \begin{tikzpicture}[scale = 2] 
    \draw[thick,->] (-1.5,0) -- (1.5,0) node[right] {};
    \draw[fill=gray!30,draw =white] (-1,2.6) -- (-1,0) arc (180:0:0.5) arc (180:0:0.5) -- (1,2.6);
    \draw[line width=0.2mm] (-1,0) -- (-1,2.6);
    \draw[line width=0.15mm,dashed] (1,0) -- (1,2.6);
    \draw[line width=0.2mm] (0,0) arc (0:180:0.5);
    \draw[line width=0.15mm,dashed] (1,0) arc (0:180:0.5);
    \draw[loosely dashed] (-1,0) arc (180:0:1);
    \draw[loosely dashed] (0,0) arc (180:90:1);
    \draw[loosely dashed] (0,0) arc (0:90:1);
    \draw[loosely dashed] (-0.5,0.5) -- (-0.5,2.6);
    \draw[loosely dashed] (0.5,0.5) -- (0.5,2.6);
    \filldraw (0,0) circle (0.75pt) node[below] {$0$};
    \filldraw (1,0) circle (0.75pt) node[below right] {$1$};
    \filldraw (-1,0) circle (0.75pt) node[below left] {$-1$};
    \filldraw (-0.5,0) circle (0.75pt) node[below] {$-\frac{1}{2}$};
    \filldraw (0.5,0) circle (0.75pt) node[below] {$\frac{1}{2}$};
    \filldraw[gray!30] (0,1.3) circle (0.75pt) node {\color{black}$\cF_{\lambda}$};
    \end{tikzpicture}
    \caption{The fundamental domain $\cF_{\lambda}$.}
    \label{fig: fundomain 2}
\end{figure} 

In the case of $\Gamma\pr{2}$ the valence formula suggests that there are $k/2$ zeros in the fundamental domain $\cF_{\lambda}$ and at the cusps.

It is possible to give an analogous definition of an Eisenstein series for $\Gamma\pr{2}$. However, we will need modular forms of a different flavor, arising from the study of elliptic functions and elliptic integrals, called the 
\emph{Jacobi theta functions}.
The Jacobi theta functions are holomorphic functions from the upper half-plane, defined as
\begin{align}
    \theta_{2}\pr{\tau}   & = \sum_{n\in\Z}q^{\pr{n+1/2}^{2}/2},  \\
    \theta_{3}\pr{\tau}   & = \sum_{n\in\Z}q^{n^{2}/2} , \quad\quad\quad\quad  q=e^{2\pi i\tau} \\ 
    \theta_{4}\pr{\tau}   & = \sum_{n\in\Z}\pr{-1}^{n}q^{n^{2}/2}. 
\end{align}
These functions are non-vanishing on $\HH$. The functions $\theta_{2}^{4}$, $\theta_{3}^{4}$, and $\theta_{4}^{4}$ are modular forms of weight $2$ for $\Gamma\pr{2}$.

We would also need the modular lambda function, defined as
\[\lambda\pr{\tau} = \frac{\theta_2^{4}\pr{\tau}}{\theta_{3}^{4}\pr{\tau}}.\]
The $q$-expansion of $\lambda$ is
\[\lambda\pr{\tau} =\sum_{n=1}^{\infty}a\pr{n}q^{n/2} = 16q^{1/2}-128q+704q^{3/2} -3072q^2 +\ldots ,\quad q=e^{2\pi i\tau}.\]
The modular lambda function is a Hauptmodul for the modular curve $X\pr{2}$, i.e., it is invariant under the action of $\Gamma\pr{2}$ on the upper half-plane and is a homeomorphism between the fundamental domain of $\Gamma\pr{2}\backslash\HH$ and $\C\smallsetminus\prs{0,1}$. It also satisfies the following transformation formulas:
\begin{align}
    \label{eq: lam inv 1}\lambda\pr{\frac{-1}{\tau}} & = 1-\lambda\pr{\tau},\\
    \label{eq: lam inv 2}\lambda\pr{\frac{1}{1-\tau}} & = \frac{1}{1-\lambda\pr{\tau}},\\
    \label{eq: lam inv 3}\lambda\pr{\frac{\tau -1}{\tau}} & = \frac{\lambda\pr{\tau}-1}{\lambda\pr{\tau}},
\end{align}
see \cite[p. 111]{ChandEllipticFunc}. We also have the value of $\lambda$ at the cusps:
\[\lambda\pr{0} = 1,\quad \lambda\pr{1}=\infty, \quad \text{and}\quad \lambda\pr{i\infty} =0.\]  Another important fact is that $\lambda$ maps the hyperbolic triangle with angles $0$ whose vertices are $0$, $1$, and $i\infty$ to the upper half-plane (see \cite[Chapter VII, p. 118, Theorem 4]{ChandEllipticFunc}).
\begin{lemma}\label{lemma: rho to rho}
    $\lambda$ has a fixed point at $\rho =e^{i\pi/3}$.
\end{lemma}
\begin{proof} 
    We have $\rho =\frac{1}{1-\rho}= \frac{\rho-1}{\rho}$, so by \eqref{eq: lam inv 2} \[\lambda\pr{\rho}=\lambda\pr{\frac{1}{1-\rho}} = \frac{1}{1-\lambda\pr{\rho}}\] and by \eqref{eq: lam inv 3}  \[\lambda\pr{\rho}=\lambda\pr{\frac{\rho -1}{\rho}} = \frac{\lambda\pr{\rho}-1}{\lambda\pr{\rho}}=\frac{-1}{\lambda\pr{\rho}^2}.\]
    Therefore, $\lambda\pr{\rho}^3 = -1$. Hence, $\lambda\pr{\rho}\in \prs{-1,\rho,\bar{\rho}}$.
    The interior of the hyperbolic triangle with angles $0$ whose vertices are $0$, $1$, and $i\infty$ is mapped to the upper-half plane under $\lambda$ and thus $\lambda\pr{\rho}\in \HH$, which yields $\lambda\pr{\rho}=\rho$.
\end{proof}
\subsection{Lattices and their associated theta functions}
Recall that a lattice in $\Rd$ is a set $\Lambda\sub \Rd$ of the form $\Lambda =g\Z^{d}$ with $g\in \GL_{d}\pr{\R}$, its covolume is $\covol\pr{\Lambda}=\abs{\det g}$ and is independent of the choice of representative $g$. The $\Z$-dual of a lattice $\Lambda$ is the set
\[\Lambda'=\prs{x\in\Rd:\ipp{x,y}\in \Z\quad\forall y\in\Lambda},\] where $\ipp{\cdot,\cdot}$ is the Euclidean inner product. A lattice is called unimodular if $\covol\pr{\Lambda}=1$. It is called integral if $\norm{x}^2 \in\Z$ for all $x\in\Lambda$, and even if $\norm{x}^2 \in 2\Z$. Finally, a lattice is called self-dual if $\Lambda=\Lambda'$. An even lattice is unimodular if and only if it is self-dual.

Theta functions, as defined in \eqref{eq: theta fn def}, have a \say{duality formula} that relates $\Theta_{\Lambda}$ to $\Theta_{\Lambda'}$ under the transformation $
\tau\mapsto\frac{-1}{\tau}$:
\begin{equation}\label{eq: theta duality}
    \Theta_{\Lambda'}\pr{-1/\tau} = \covol\pr{\Lambda}\pr{-i\tau}^{d/2}\Theta_{\Lambda}\pr{\tau}.
\end{equation}
The duality formula connects theta functions to modular forms. As previously stated, when the lattice $\Lambda \sub\Rd$ is integral and self-dual, the theta function $\Theta_{\Lambda}$ is a modular form of weight $d/2$ of level $2$ and if $\Lambda$ is also even, $\Theta_{\Lambda}$ is a modular form of weight $d/2$ for $\slz$. 

\section{The structure of $\Theta_{\Gamma_{8k}}$ and their $\lambda$-zeros}\label{sec: lambda conf pro}
We begin by studying the structure of $\Theta_{\Gamma_{8k}}$, showing two representations of $\Theta_{\Gamma_{8k}}$.
\subsection{A couple representations of $\Theta_{\Gamma_{8k}}$}
Our first representation connects $\Theta_{\Gamma_{8k}}$ to Jacobi theta functions.
\begin{lemma}\label{prop: General Jacobi formula}
    For any $n\ge 1$, we have 
    \begin{equation}\label{eq: jacobi rep}
        \Theta_{\Gamma_{n}} =\frac{1}{2}\pr{ \theta_{2}^{n}+\theta_{3}^{n}+\theta_{4}^{n}}.
    \end{equation}
\end{lemma}
\begin{proof} This is a known fact; see for example \cite[p. 119-120, eq. $(94)$]{ConwaySloane}. The proof is straightforward; however, as we could not locate the full derivation in any standard text, we provide the details in the appendix.
\end{proof}

A well-known identity, which is due to Jacobi (see \cite[p.\ 103, eq.\ (3.10)]{ChandEllipticFunc}), gives us a parametrization of the Fermat curve in terms of Jacobi theta functions:
\begin{equation}\label{eq: Jacobi's identity}
    \theta_{3}^{4}=\theta_{2}^{4}+\theta_{4}^{4}.
\end{equation}
Using \eqref{eq: Jacobi's identity} and factoring $\theta_{3}^{4k}$ from the right-hand side of \eqref{eq: jacobi rep} yields
\begin{equation}\label{eq: lambda rep}
    \Theta_{\Gamma_{4k}} =\frac{\theta_3^{4k}}{2}\pr{1+\lambda^{k}+\pr{1-\lambda}^{k}}.
\end{equation}
The function $\theta_{3}^{4k}$ is non-vanishing on $\HH$. Hence, the zeros of $\Theta_{\Gamma_{4k}}$ are the \say{$\lambda$-zeros} of the polynomial $p_{k}$, where we denote $p_{k}\pr{z}=1+z^k+\pr{1-z}^k$. Our goal now is to study the roots of the polynomial $p_{k}$ and their pre-image under $\lambda$. The polynomials $p_{k}$ are closely tied to the Cauchy-Mirimanoff polynomials, whose Galois group and roots were studied by Helou \cite{Helou} and later by Nanninga \cite{Nanninga}. They display a strong pattern (see Figure \ref{fig: poly roots}).

\begin{figure}[ht]
    \centering
    \begin{subfigure}[b]{0.3\textwidth}
    \begin{tikzpicture}[scale = 1] 
        \definecolor{newblue}{RGB}{26, 128, 187};
        \draw[gray,thick,->] (-1.2,0) -- (2.2,0) node[right] {};
        \draw[gray,thick,->] (0,-3) -- (0,3) node[right] {};
        \draw[gray, line width = 0.2mm] (0.5,-3) -- (0.5,3);
        \draw[line width=0.2mm, gray] (1,0) circle (1);
        \draw[line width=0.2mm, gray] (0,0) circle (1);
        \filldraw[gray] (2,0) circle (0.75pt) node[below right] {\tiny$2$};
        \filldraw[gray] (1,0) circle (0.75pt) node[below right] {\tiny$1$};
        \filldraw[gray] (-1,0) circle (0.75pt) node[below left] {\tiny$-1$};
        \filldraw[gray] (0,0) circle (0.75pt) node[below left] {\tiny$0$};
        \filldraw[gray] (0,2) circle (0.75pt) node[ left] {\tiny$2$};
        \filldraw[gray] (0,1) circle (0.75pt) node[above left] {\tiny$1$};
        \filldraw[gray] (0,-1) circle (0.75pt) node[below left] {\tiny$-1$};
        \filldraw[gray] (0,-2) circle (0.75pt) node[ left] {\tiny$-2$};
        \filldraw[gray] (-60:1) circle (0.75pt) node[below right] {\tiny$\bar{\rho}$};
        \filldraw[gray] (60:1) circle (0.75pt) node[above right] {\tiny$\rho$};
        \foreach \point in {(0.5, 0.866025), (0.5, -0.866025), (0.5, -2.08265), (0.5, 2.08265), (0.5, -1.2067), (0.5, 1.2067), (0.293062, 0.707276), (0.706938, -0.707276), (0.293062, -0.707276), (0.706938, 0.707276), (0.108993, 0.45399), (0.891007, -0.45399), (0.108993, -0.45399), (0.891007, 0.45399), (0.0123117, 0.156434), (0.987688, -0.156434), (0.0123117, -0.156434), (0.987688,0.156434)}{
            \filldraw[newblue] \point circle (1pt);
        }
    \end{tikzpicture}
    \end{subfigure}
    \begin{subfigure}[b]{0.3\textwidth}
    \begin{tikzpicture}[scale = 1] 
        \definecolor{newblue}{RGB}{26, 128, 187};
        \draw[gray,thick,->] (-1.2,0) -- (2.2,0) node[right] {};
        \draw[gray,thick,->] (0,-3) -- (0,3) node[right] {};
        \draw[gray, line width = 0.2mm] (0.5,-3) -- (0.5,3);
        \draw[line width=0.2mm, gray] (1,0) circle (1);
        \draw[line width=0.2mm, gray] (0,0) circle (1);
        \filldraw[gray] (2,0) circle (0.75pt) node[below right] {\tiny$2$};
        \filldraw[gray] (1,0) circle (0.75pt) node[below right] {\tiny$1$};
        \filldraw[gray] (-1,0) circle (0.75pt) node[below left] {\tiny$-1$};
        \filldraw[gray] (0,0) circle (0.75pt) node[below left] {\tiny$0$};
        \filldraw[gray] (0,2) circle (0.75pt) node[ left] {\tiny$2$};
        \filldraw[gray] (0,1) circle (0.75pt) node[above left] {\tiny$1$};
        \filldraw[gray] (0,-1) circle (0.75pt) node[below left] {\tiny$-1$};
        \filldraw[gray] (0,-2) circle (0.75pt) node[ left] {\tiny$-2$};
        \filldraw[gray] (-60:1) circle (0.75pt) node[below right] {\tiny$\bar{\rho}$};
        \filldraw[gray] (60:1) circle (0.75pt) node[above right] {\tiny$\rho$};
        
        \foreach \point in {(0.5, 0.8660254037844386), (0.5, -0.8660254037844386), (0.5, -2.236871414605777), (0.5, 2.236871414605777), (0.5, -1.7210112883346231), (0.5, 1.7210112883346231), (0.5, -1.3888034268047544), (0.5, 1.3888034268047544), (0.5, -1.1554321432952352), (0.5, 1.1554321432952352), (0.5, -0.9811093281559106), (0.5, 0.9811093281559106), (0.4123454533856929, 0.809111941478763), (0.587654546614307, -0.809111941478763), (0.4123454533856929, -0.809111941478763), (0.587654546614307, 0.809111941478763), (0.31545274857681005, 0.7289684907929533), (0.68454725142319, -0.7289684907929533), (0.31545274857681005, -0.7289684907929533), (0.68454725142319, 0.7289684907929533), (0.22948675726889445, 0.637423989802703), (0.7705132427311056, -0.637423989802703), (0.22948675726889445, -0.637423989802703), (0.7705132427311056, 0.637423989802703), (0.15567207449798265, 0.535826794978993), (0.8443279255020173, -0.535826794978993), (0.15567207449798265, -0.535826794978993), (0.8443279255020173, 0.535826794978993), (0.09517294753398042, 0.42577929156507255), (0.9048270524660196, -0.42577929156507255), (0.09517294753398042, -0.42577929156507255), (0.9048270524660196, 0.42577929156507255), (0.04894348370484647,   0.30901699437494734), (0.9510565162951535, -0.30901699437494734), (0.04894348370484647, -0.30901699437494734), (0.9510565162951535, 0.30901699437494734), (0.01771274927131128, 0.18738131458572463), (0.9822872507286887, -0.18738131458572463), (0.01771274927131128, -0.18738131458572463), (0.9822872507286887, 0.18738131458572463),(0.001973271571728441, 0.06279051952931339), (0.9980267284282716, -0.06279051952931339), (0.001973271571728441, -0.06279051952931339),(0.9980267284282716, 0.06279051952931339)}{
            \filldraw[newblue] \point circle (1pt);
        }
    \end{tikzpicture}
    \end{subfigure}
    \begin{subfigure}[b]{0.3\textwidth}
    \begin{tikzpicture}[scale = 1] 
        \definecolor{newblue}{RGB}{26, 128, 187};
        \draw[gray,thick,->] (-1.2,0) -- (2.2,0) node[right] {};
        \draw[gray,thick,->] (0,-3) -- (0,3) node[right] {};
        \draw[gray, line width = 0.2mm] (0.5,-3) -- (0.5,3);
        \draw[line width=0.2mm, gray] (1,0) circle (1);
        \draw[line width=0.2mm, gray] (0,0) circle (1);
        \filldraw[gray] (2,0) circle (0.75pt) node[below right] {\tiny$2$};
        \filldraw[gray] (1,0) circle (0.75pt) node[below right] {\tiny$1$};
        \filldraw[gray] (-1,0) circle (0.75pt) node[below left] {\tiny$-1$};
        \filldraw[gray] (0,0) circle (0.75pt) node[below left] {\tiny$0$};
        \filldraw[gray] (0,2) circle (0.75pt) node[ left] {\tiny$2$};
        \filldraw[gray] (0,1) circle (0.75pt) node[above left] {\tiny$1$};
        \filldraw[gray] (0,-1) circle (0.75pt) node[below left] {\tiny$-1$};
        \filldraw[gray] (0,-2) circle (0.75pt) node[ left] {\tiny$-2$};
        \filldraw[gray] (-60:1) circle (0.75pt) node[below right] {\tiny$\bar{\rho}$};
        \filldraw[gray] (60:1) circle (0.75pt) node[above right] {\tiny$\rho$};
        
        \foreach \point in {(0.5, 0.866025), (0.5, 0.866025), (0.5, -0.866025), (0.5, -0.866025), (0.5, -2.86487), (0.5, 2.86487), (0.5, -2.41441), (0.5, 2.41441), (0.5, -2.08265), (0.5, 2.08265), (0.5, -1.82769), (0.5, 1.82769), (0.5, -1.62528), (0.5, 1.62528), (0.5, -1.46038), (0.5, 1.46038), (0.5, -1.32321), (0.5, 1.32321), (0.5, -1.20711), (0.5, 1.20711), (0.5, -1.10738), (0.5, 1.10738), (0.5, -1.02063), (0.5, 1.02063), (0.5, -0.944351), (0.5, 0.944351), (0.437906, 0.827073), (0.562094, -0.827073), (0.437906, -0.827073), (0.562094, 0.827073), (0.387093, 0.790155), (0.612907, -0.790155), (0.387093, -0.790155), (0.612907, 0.790155), (0.338688, 0.750111), (0.661312, -0.750111), (0.338688, -0.750111), (0.661312, 0.750111), (0.292893, 0.707107), (0.707107, -0.707107), (0.292893, -0.707107), (0.707107, 0.707107), (0.249889, 0.661312), (0.750111, -0.661312), (0.249889, -0.661312), (0.750111, 0.661312), (0.209845, 0.612907), (0.790155, -0.612907), (0.209845, -0.612907), (0.790155, 0.612907), (0.172919, 0.562083), (0.827081, -0.562083), (0.172919, -0.562083), (0.827081, 0.562083), (0.139258, 0.509041), (0.860742, -0.509041), (0.139258, -0.509041), (0.860742, 0.509041), (0.108993, 0.45399), (0.891007, -0.45399), (0.108993, -0.45399), (0.891007, 0.45399), (0.0822454, 0.397148), (0.917755, -0.397148), (0.0822454, -0.397148), (0.917755,0.397148), (0.0591192, 0.338738), (0.940881, -0.338738), (0.0591192, -0.338738), (0.940881,0.338738), (0.0397063, 0.278991), (0.960294, -0.278991), (0.0397063, -0.278991), (0.960294,0.278991), (0.0240832, 0.218143), (0.975917, -0.218143), (0.0240832, -0.218143), (0.975917,0.218143), (0.0123117, 0.156434), (0.987688, -0.156434), (0.0123117, -0.156434), (0.987688,0.156434), (0.00443804, 0.0941083), (0.995562, -0.0941083), (0.00443804, -0.0941083), (0.995562, 0.0941083), (0.00049344, 0.0314108), (0.999507, -0.0314108), (0.00049344, -0.0314108), (0.999507, 0.0314108)}{
            \filldraw[newblue] \point circle (1pt);
        }
    \end{tikzpicture}
    \end{subfigure}
    \caption{The roots of $p_{k}$ where $k\in\prs{20,50,100}$}
    \label{fig: poly roots}
\end{figure}

We will now prove several propositions: First, we will study the zeros of the polynomials $p_{k}$. Then, we will exploit conformal properties of $\lambda$ to study how the geodesics $\cL_{\rho}$, $\cL_{\rho}^*$, $\cC$, $\cC^{*}$, $\cU$, and $\cU^*$ behave under $\lambda$.
\subsection{The zeros of the polynomial $p_{k}$} Recall we denote
\[p_{k}\pr{z}=1+z^k+\pr{1-z}^k.\]
We will need the auxiliary polynomial \[q_{k}\pr{z}=1+z^k +\pr{z-1}^{k} = 1+z^{k}+\pr{-1}^{k}\pr{1-z}^{k}.\]
The polynomials $p_{k}$ and $q_{k}$ satisfy:
    \begin{equation}\label{eq: poly stable 1}
        p_{k}\pr{\frac{z-1}{z}} = 1 + \pr{\frac{z-1}{z}}^{k} + \pr{1-\frac{z-1}{z}}^{k} =
       \frac{z^{k}+\pr{z-1}^{k} + 1}{z^{k}} = \frac{q_{k}\pr{z}}{z^k}.
    \end{equation}

We will show that $q_{k}$ has at least $\floor{\frac{k}{2}}-\ceil{\frac{k}{3}}$ zeros on the arc $\cC = \prs{1+e^{i\varphi}:\varphi\in\left(\frac{2\pi}{3},\pi\right)}$, and that the zeros are equidistributed.
\begin{remark}
    The results regarding the location of the zeros of $p_{k}$ and $q_{k}$ are essentially known; they follow from the results on the Cauchy-Mirimanoff polynomials proven by Helou and later by Nanninga in \cite{Helou, Nanninga} respectively. However, we state and prove them for completeness, as we rely heavily on the lemmata below.
\end{remark}

\begin{prop}\label{prop: q polyzeros}
    For any $k\ge 1$, the polynomial $q_{k}\pr{z}=1+z^{k}+\pr{z-1}^{k}$ has at least $d=\floor{\frac{k}{2}}-\ceil{\frac{k}{3}}$ zeros on the arc $\cC = \prs{1+e^{i\varphi}:\varphi\in\pr{\frac{2\pi}{3},\pi}}$. Furthermore, there exist $\varphi_{k,1},\ldots ,\varphi_{k,d} \in \pr{\frac{2\pi}{3},\pi}$ such that $q_{k}\pr{1+e^{i\varphi_{j,k}}}=0$ for any $1\le j\le d$, and: 
    \begin{enumerate}[label = (\roman*)]
        \item For any $1\le j\le d$ we have $\varphi_{k,j}\in\pr{\frac{2\pi}{k}\pr{\ceil{\frac{k}{3}}+j-1},\frac{2\pi}{k}\pr{\ceil{\frac{k}{3}}+j}}$.
        \item The zeros become equidistributed on $\text{cl}\pr{\cC}$ as $k\to\infty$, i.e., for all $\prb{a,b}\sub\prb{\frac{2\pi}{3},\pi}$ we have 
        \[\frac{\#\prs{1\le j\le d:\varphi_{j,k}\in\prb{a,b}}}{d} \xrightarrow[]{k\to\infty}\frac{{b-a}}{\pi-\frac{2\pi}{3}}.\]
    \end{enumerate}
\end{prop}
\begin{proof}
First, denote
\begin{multline}\label{eq: aux fn}
    f_{k}\pr{\varphi}=e^{-ik\varphi/2}q_{k}\pr{1+e^{i\varphi}}\\=e^{-ik\varphi/2}+e^{-ik\varphi/2}\pr{1+e^{i\varphi}}^{k}+e^{-ik\varphi/2}\pr{\pr{1+e^{i\varphi}}-1}^{k}= 2\cos\pr{k\varphi/2}+\pr{2\cos\pr{\varphi/2}}^{k}.
\end{multline}
Therefore, $f_{k}$ is real-valued and continuous. 
For any $\varphi\in \prb{\frac{2\pi}{3},\pi}$, we have $0\le \cos\pr{\varphi/2}\le \frac{1}{2}$. 
Hence, for all  $\varphi\in \prb{\frac{2\pi}{3},\pi}$ we have
\begin{equation}\label{eq: ineq almost cos}
    \abs{f_{k}\pr{\varphi}-2\cos\pr{k\varphi/2}}= \abs{2\cos\pr{\varphi/2}}^{2k} \le 1.
\end{equation}
As $\varphi$ increases from $\frac{2\pi}{3}$ to $\pi$, the parameter $k\varphi$ passes through exactly $\floor{\frac{k}{2}} -\ceil{\frac{k}{3}} +1 = d+1$ integer multiples of $\pi$. The least integer multiple of $\pi$ in the interval $\prb{\pi k/3,\pi k/2}$ is $\pi\ceil{\frac{k}{3}}$. Let $\pi r$ and $\pi \pr{r+1}$ be consecutive multiples of $\pi$ in $\prb{\pi k/3,\pi k/2}$. If $r$ is even, by \eqref{eq: ineq almost cos} we have 
\[f_{k}\pr{\frac{2\pi r}{k}} \ge 2\cos\pr{\frac{k}{2}\cdot\frac{2\pi r}{k}}-1=2\pr{-1}^r +1 =1\]
and \[f_{k}\pr{\frac{2\pi\pr{r+1}}{k}} \le 2\cos\pr{\frac{k}{2}\cdot\frac{2\pi\pr{r+1}}{k}}+1=2\pr{-1}^{r+1} +1 = -1.\]
Similarly, if $r$ is odd, we have $f_{k}\pr{\frac{\pi r}{k}} < -1$ and $f_{k}\pr{\frac{\pi\pr{r+1}}{k}}>1$. In any case, by the intermediate value theorem, there exists $\varphi\in\pr{{\frac{\pi{r}}{k}},{\frac{\pi\pr{r+1}}{k}}}$ such that $f_{k}\pr{\varphi} =0$, and therefore $q_{k}\pr{1+e^{i\varphi}}=0$.
Hence, for any $1\le j\le d$ there exists $\varphi_{k,j}\in\pr{\frac{2\pi}{k}\pr{\ceil{\frac{k}{3}}+j-1},\frac{2\pi}{k}\pr{\ceil{\frac{k}{3}}+j}}$ such that $q_{k}\pr{1+e^{i\varphi_{k,j}}}=0$.

We have proven (i), as required; we are left to prove (ii).
Let $\prb{a,b}\sub\prb{\frac{2\pi}{3},\pi}$. When $\varphi$ increases from $a$ to $b$, the parameter $k\varphi/2$ increases from $ka/2$ to $kb/2$ and passes through exactly $\floor{\frac{kb}{2\pi}}-\ceil{\frac{ka}{2\pi}} +1$ integer multiples of $\pi$. The number of $\varphi_{k,j}$ in the interval $\prb{a,b}$ is one less than the number of sign changes in the interval, which is exactly $\floor{\frac{kb}{2\pi}}-\ceil{\frac{ka}{2\pi}}$, since $\varphi_{k,j} \in \prb{a,b}$ if and only if $k\varphi_{k,j}/2\in\prb{ka/2, kb/2}$.
Thus,
\begin{equation*}
    \#\prs{1\le j\le d:\varphi_{k,j}\in\prb{a,b}} =\floor{\frac{kb}{2\pi}}-\ceil{\frac{ka}{2\pi}}={\frac{k\pr{b-a}}{2\pi}}+O\pr{1}.    
\end{equation*}
Also note that $ d = \floor{\frac{k}{2}}-\ceil{\frac{k}{3}} = \frac{k}{6}\pr{1+O\pr{\frac{1}{k}}}$.
Hence,
\begin{multline*}
    \frac{\#\prs{1\le j\le d:\varphi_{k,j}\in\prb{a,b}}}{d} = \frac{k\pr{b-a}}{2\pi d}+O\pr{\frac{1}{d}} \\ =\frac{k}{2\pi \frac{k}{6}\pr{1+O\pr{\frac{1}{k}}}}\pr{b-a}+O\pr{\frac{1}{k}} \xrightarrow[]{k\to\infty}\frac{3\pr{b-a}}{\pi}=\frac{b-a}{\pi-\frac{2\pi}{3}}.
\end{multline*}
Therefore, the zeros become equidistributed on $\cC$ as $k\to\infty$.
\end{proof}

For the even case, the proposition above is all we need, as $p_{2k}=q_{2k}$.
In fact, a corollary of Proposition \ref{prop: q polyzeros} is that the zeros of $p_{2k}$ are always on the line $\Re z=\frac{1}{2}$ and are on the arc $\prs{e^{i\varphi}:\varphi\in\prb{-\frac{\pi}{3},\frac{\pi}{3}}}$, and the arc $\prs{1+e^{i\varphi}:\varphi\in\prb{\frac{2\pi}{3},\frac{4\pi}{3}}}$. While we do not use this in our argument for the even case, the following proposition provides valuable context for the odd case and an explanation for the images we see in Figure \ref{fig: poly roots}:
\begin{prop}\label{prop: full polyzeros}
     The polynomial $p_{2k}$ has $\ell = k-\ceil{\frac{2k}{3}} = \floor{\frac{k}{3}}$ simple zeros on each of $\cL_{\rho}$, $\overline{\cL_{\rho}}$, $\cC$, $\overline{\cC}$,  $\cU$, and $\overline{\cU}$. Additionally, there exists a zero of multiplicity $\frac{k'}{4}$ at $\rho$ and $\bar{\rho}$. Here $\cL_{\rho}$, $\cC$, and $\cU$ are as defined in \eqref{eq: line 0.5}, \eqref{eq: the arc center 1}, and \eqref{eq: unit arc}, $\overline{A}=\prs{\bar{w}:w\in A}$, and we write $4k=12\ell + k'$. 
\end{prop}

\begin{proof}
    Recall that by \eqref{eq: poly stable 1} for all $z\in\C\smallsetminus\prs{0}$  
    \begin{equation}\label{eq: poly stable even}
        p_{2k}\pr{\frac{z-1}{z}} = \frac{q_{2k}\pr{z}}{z^{2k}} = \frac{p_{2k}\pr{z}}{z^{2k}},
    \end{equation}
    Under the map $z\mapsto \frac{z-1}{z}$, the arc $\cC$ maps to the line $\cL_{\rho}$ (see Figure \ref{fig: Mobius arc C to line}).
    \begin{figure}[ht]
    \centering
    \begin{subfigure}[b]{0.4\textwidth}
    \begin{tikzpicture}[scale = 1.3] 
        \definecolor{newblue}{RGB}{26, 128, 187};
        \definecolor{newred}{RGB}{193, 39, 45};
        \draw[gray,thick,->] (-1.2,0) -- (2.2,0) node[right] {};
        \draw[gray,thick,->] (0,-1.2) -- (0,2.2) node[right] {};
        \draw[gray, line width = 0.2mm] (0.5,-1.2) -- (0.5,2.2);
        \draw[line width=0.2mm, gray] (1,0) circle (1);
        \draw[line width=0.2mm, gray] (0,0) circle (1);
        \draw[line width =0.3mm, newred] (0,0) arc (180:120:1);
        \draw[thin, newred,->] ([shift=(140:0.9)]1,0) arc (140:160:0.9);
        \filldraw[gray] (2,0) circle (0.75pt) node[below right] {\tiny$2$};
        \filldraw[gray] (1,0) circle (0.75pt) node[below right] {\tiny$1$};
        \filldraw[gray] (-1,0) circle (0.75pt) node[below left] {\tiny$-1$};
        \filldraw[gray] (0,0) circle (0.75pt) node[below left] {\tiny$0$};
        \filldraw[gray] (0,2) circle (0.75pt) node[ left] {\tiny$2$};
        \filldraw[gray] (0,1) circle (0.75pt) node[above left] {\tiny$1$};
        \filldraw[gray] (0,-1) circle (0.75pt) node[below left] {\tiny$-1$};
        \filldraw[gray] (-60:1) circle (0.75pt) node[below right] {\tiny$\bar{\rho}$};
        \filldraw[gray] (60:1) circle (0.75pt) node[above right] {\tiny$\rho$};
        \node at (2.5,0.5)  [above]{$z\mapsto\frac{z-1}{z}$};
    \end{tikzpicture}
    \end{subfigure}
    \begin{subfigure}[b]{0.4\textwidth}
    \begin{tikzpicture}[scale = 1.3] 
        \definecolor{newblue}{RGB}{26, 128, 187};
        \definecolor{newred}{RGB}{193, 39, 45};
        \draw[gray,thick,->] (-1.2,0) -- (2.2,0) node[right] {};
        \draw[gray,thick,->] (0,-1.2) -- (0,2.2) node[right] {};
        \draw[gray, line width = 0.2mm] (0.5,-1.2) -- (0.5,2.2);
        \draw[line width = 0.3mm,newred] (60:1) -- (0.5,2.2);
        \draw[thin ,newred,->] (0.4,1.35) -- (0.4,1.65);
        \draw[line width=0.2mm, gray] (1,0) circle (1);
        \draw[line width=0.2mm, gray] (0,0) circle (1);
        \filldraw[gray] (2,0) circle (0.75pt) node[below right] {\tiny$2$};
        \filldraw[gray] (1,0) circle (0.75pt) node[below right] {\tiny$1$};
        \filldraw[gray] (-1,0) circle (0.75pt) node[below left] {\tiny$-1$};
        \filldraw[gray] (0,0) circle (0.75pt) node[below left] {\tiny$0$};
        \filldraw[gray] (0,2) circle (0.75pt) node[ left] {\tiny$2$};
        \filldraw[gray] (0,1) circle (0.75pt) node[above left] {\tiny$1$};
        \filldraw[gray] (0,-1) circle (0.75pt) node[below left] {\tiny$-1$};
        \filldraw[gray] (-60:1) circle (0.75pt) node[below right] {\tiny$\bar{\rho}$};
        \filldraw[gray] (60:1) circle (0.75pt) node[above right] {\tiny$\rho$};
    \end{tikzpicture}
    \end{subfigure}
    \caption{The mapping of the arc $\cC$ under the map $z\mapsto\frac{z-1}{z}$.}
    \label{fig: Mobius arc C to line}
\end{figure}
By \eqref{eq: poly stable even}, the zero set of $p_{2k}$ is stable under the transformation $z\mapsto\frac{z-1}{z}$, and since for $q_{2k}=p_{2k}$ there exist $\ell$ zeros on the arc $\cC$ by Proposition \ref{prop: q polyzeros}, there exist $\ell$ zeros on the line $\cL_{\rho}$. 

We also have $p_{2k}\pr{1-z} = p_{2k}\pr{z}$ so the zero set of $p_{2k}$ is stable under the transformation $z\mapsto 1-z$. Under the the transformation $z\mapsto 1-z$, the arc $\cC$ is mapped to $\cU$. Again, since there are $\ell$ zeros on the arc $\cC$, we can deduce that there are $\ell$ zeros on the arc $\cU$.

In addition, $p_{2k}$ has real coefficients, and the zeros on the arcs and lines above are non-real. Hence, by conjugating the zeros, there exist $\ell$ zeros on each of the arcs  $\overline{\cL_{\rho}}$, $\overline{\cC}$, and $\overline{\cU}$. Lastly, we have 
\[p_{2k}\pr{\rho} = 1+ \rho^{2k}+\pr{1-\rho}^{2k} = 1+2\cos\pr{\frac{2\pi k}{3}}=\begin{cases}
    3, & k'=0,\\
    0, & k'=4,8,
\end{cases}\]
and 
\[p_{2k}'\pr{\rho} = 2k\rho^{2k-1}-2k\pr{1-\rho}^{2k-1} = i4k\sin\pr{\frac{2\pi k}{3}-\frac{\pi}{3}}=\begin{cases}
    -i2\sqrt{3}k, & k'=0,\\
    i2\sqrt{3}k, & k'=4,\\
    0, & k'=8.
\end{cases}\]
Hence, there exists a zero of multiplicity $\frac{k'}{4}$ at $\rho$, and by conjugation, at $\bar{\rho}$ as well.
Finally, $2k = 6\ell +\frac{k'}{4}+\frac{k'}{4}$, and therefore, the zeros we found above account for all of the $2k$ zeros of $p_{2k}$. 
\end{proof}

We now turn to the odd case:

From the proof we get that $p_{2k}$ has simple zeros except for $\rho$, the same is true for odd $p_{k}$ with odd $k$:
\begin{lemma}\label{lemma: simplicity}
    For any odd $k\ge 3$, the zeros of $p_{k}$ are simple.
\end{lemma}
\begin{proof}
    First, observe that \[ p_k\pr{\rho} = p_{k}\pr{\bar{\rho}} = 1+ \rho^{k}+\pr{1-\rho}^{k} = 1+2\cos\pr{\frac{\pi k}{3}} = 0 \iff k\equiv2,4\bmod{6},\]
    in particular, $k$ is even. Now, suppose $p_{k}$ has a non-simple zero and let $w\in \C$ be a non-simple zero of $p_{k}$, i.e. $p_{k}\pr{w} =p_{k}'\pr{w}= 0$. We claim that $w\in\prs{\rho,\bar{\rho}}$:
    Indeed, we have 
    \[0=p_{k}'\pr{w}=kw^{k-1}-k\pr{1-w}^{k-1}.\]
    Solving the equation above, we get $w = \frac{1}{1+e^{\frac{2\pi ij}{k-1}}}$ for some $j\in \Z$ which shows $\Re\pr{w} = \frac{1}{2}$. On the other hand, $\abs{w} = 1$: We have \[0=p_{k}\pr{z}=1+w^k + \pr{1-w}^{k} = 1+ w^{k-1}\pr{w+\pr{1-w}} = 1+w^{k-1}.\] Therefore, $\abs{w} = 1$ and $\Re w= \frac{1}{2}$ which shows $w\in \prs{\rho,\bar{\rho}}$ and by our observation $k$ is even. Hence, for odd $k$, any zero of $p_{k}$ must be simple.
\end{proof}

Another result in the odd case, that resembles Proposition \ref{prop: full polyzeros}, is that there are always $\ell$ zeros on the line $\Re z =\frac{1}{2}$. However, as we will soon see, the zeros of $p_{2k+1}$ are not on the arcs as in the even case, only exponentially close.
We begin with the following: 
\begin{lemma}\label{lemma: thm 1.3 lemma}
    For any $k\ge 1$, for the polynomial $p_{2k+1}$, there are at least $\ell$ distinct zeros on each of the lines $\cL_{\rho}$ and $\overline{\cL_{\rho}}$, and no zeros on the arcs $\cC$, $\overline{\cC}$, $\cU$, and $\overline{\cU}$. Here, we write $4k+2 = 12\ell + k'$ with $k'\in\prs{6,10,14}$.
\end{lemma}

\begin{proof}
As in the proof of Proposition \ref{prop: full polyzeros}, by \eqref{eq: poly stable 1}, for all $z\in\C\smallsetminus\prs{0}$  
    \begin{equation}\label{eq: poly stable odd}
        p_{2k+1}\pr{\frac{z-1}{z}} = \frac{q_{2k+1}\pr{z}}{z^{2k+1}},
    \end{equation}
    and recall that the line $\cL_{\rho}$ is the image of the arc $\cC$ under the map $z\mapsto \frac{z-1}{z}$. Hence, any zero of $q_{k}$ on the arc $\cC$ accounts for exactly one zero on the line $\cL_{\rho}$ as the map $z\mapsto \frac{z-1}{z}$ is injective. By proposition \ref{prop: q polyzeros}, the polynomial $q_{2k+1}$ has $\floor{\frac{2k+1}{2}} - \ceil{\frac{2k+1}{3}} = \ell$ distinct zeros on $\cC$, thus, $p_{2k+1}$ has at least $\ell$ distinct zeros on $\cL_{\rho}$. The same is true for $\overline{\cL_{\rho}}$ as the coefficients of $p_{k}$ are real, and the zero set is stable under conjugation.

    To show that there are no zeros on the arcs $\cC$, $\overline{\cC}$, $\cU$, and $\overline{\cU}$ it suffices to show that there are no zeros on the arc $\cC$, as the the zero set is stable under conjugation and reflection along the line $\Re z=\frac{1}{2}$. 
    
    Consider the function \[f_{k}\pr{\varphi} = e^{-i\pr{2k+1}\frac{\varphi}{2}}p_{2k+1}\pr{1+e^{i\varphi}} = 2i\sin\pr{\frac{2k+1}{2}\varphi}+\pr{2\cos\pr{\frac{\varphi}{2}}}^{2k+1},\]
    for $\varphi\in\prb{\frac{2\pi}{3},\pi}$. The right summand on the right-hand side is non-vanishing on $\left[\frac{2\pi}{3},\pi\right)$ and the left summand on the right-hand side is non-vanishing at $\varphi =\pi$. Hence, $f_{k}$ does not vanish on $\prb{\frac{2\pi}{3},\pi}$ and thus $p_{2k+1}$ does not vanish on $\cC$.
\end{proof}

Even though $p_{2k+1}$ has no zero on the arc $\cC$, we can give an explicit exponential bound for large values of $k$:

\begin{prop}\label{prop: poly odd decay}
    For all $a\in \pr{\frac{2\pi}{3},\pi}$, for all  $k\gg 1$ and any $j\in \Z $ such that $a+\frac{\pi}{8k+4}\le \frac{2\pi j}{2k+1}\le\pi$, there exists a unique $z_{j}\in\pr{\frac{2\pi j}{2k+1}-\frac{\pi}{8k+4},\frac{2\pi j}{2k+1}+\frac{\pi}{8k+4}}\times\pr{-\frac{2}{2k+1},\frac{2}{2k+1}}$ such that $p_{2k+1}\pr{z_j}=0$. Moreover, \[\abs{z_j - \frac{2\pi j}{2k+1}}\le \frac{\sqrt{2}}{2k+1}e^{-ck},\]
    where $c=2\log\pr{\frac{1+2\cos\pr{\frac{a}{2}}}{2}}$.
\end{prop}
\begin{proof}
Consider the function
\[f_{k}\pr{z} = e^{-i\pr{2k+1}\frac{z}{2}}p_{2k+1}\pr{1+e^{iz}} = 2i\sin\pr{\frac{2k+1}{2}z}+\pr{2\cos\pr{\frac{z}{2}}}^{2k+1}.\]
Our goal is to show that the zeros of $f_{k}$ are close to the zeros of $2i\sin\pr{\frac{2k+1}{2}z}$ on rectangles of the form \[\pr{\frac{2\pi j}{2k+1}-\frac{\pi}{8k+4},\frac{2\pi j}{2k+1}+\frac{\pi}{8k+4}}\times\pr{-\frac{2}{2k+1},\frac{2}{2k+1}}\]
which we will denote $R_{k,j}$ for brevity. 

For all $z,w\in \C$ we have \[\abs{e^{iz}-e^{iw}}^2 = e^{-2\Im z} + e^{-2\Im w} -2e^{-\Im z -\Im w}\cos\pr{\Re z - \Re w}.\]
Therefore, we have 
\begin{equation}\label{eq: abs(sin)}
    \abs{2i\sin\pr{z}}^2 = 2\cosh\pr{2\Im z} -2\cos\pr{2\Re z}
\end{equation} and for all $z\in \C$ such that $\abs{\Im z}\le1$ we have 
\begin{align*}
    \abs{2\cos\pr{z}-2\cos\pr{\Re z}} & = \abs{e^{iz}-e^{i\Re z} + e^{-iz}-e^{-i\Re z}}\\
     & \le \abs{e^{iz}-e^{i\Re z}} + \abs{e^{-iz}-e^{-i\Re z}} \\
     & = \abs{e^{-\Im z}-1} + \abs{e^{\Im z}-1} \\
     & = 2\sinh\abs{\Im z} \\
     & \le 2\cosh\pr{1}\abs{\Im z}.
\end{align*}
where the last inequality is true for any $z\in \prb{a,\pi}\times\prb{-1,1}$.
Hence, for all $z\in \prb{a,\pi}\times\prb{-1,1}$ we have 
\begin{equation}\label{eq: cosine bound}
    \abs{2\cos\pr{\frac{z}{2}}} \le \cosh\pr{1}\abs{\Im\pr{z}}+\abs{2\cos\pr{\frac{\Re z}{2}}} \le \cosh\pr{1}\abs{\Im z} + 2\cos\pr{\frac{a}{2}}.
\end{equation}

Let $a\in \pr{\frac{2\pi}{3},\pi}$, then $1>2\cos\pr{\frac{a}{2}}>0$. Hence, there exists $K\in\N$  such that for all $k\ge K$ we have $\frac{2}{2k+1}<\frac{1-2\cos\pr{\frac{a}{2}}}{2\cosh\pr{1}}$ and $\pr{\frac{1+2\cos\pr{\frac{a}{2}}}{2}}^{2k+1}<\sqrt{2-\sqrt{2}}$.

Let $k\ge K$, and let $j\in \Z $ be such that $a+\frac{\pi}{8k+4}\le \frac{2\pi j}{2k+1}\le\pi$. We will use Rouch{\'e}'s theorem to show there exists a unique zero of $f_{k}$ in $R_{k,j}$, i.e., we will show that 
\[\abs{f_{k}\pr{z}-2i\sin\pr{\frac{2k+1}{2}z}} = \abs{2\cos\pr{\frac{z}{2}}}^{2k+1}<\abs{2i\sin\pr{\frac{2k+1}{2}z}}\]
for all $z\in\partial R_{k,j}$. Let $z\in\partial R_{k,j}$, by \eqref{eq: cosine bound} we have 
\begin{multline*}
    \abs{2\cos\pr{\frac{z}{2}}}\le \cosh\pr{1}\abs{\Im z} + 2\cos\pr{\frac{a}{2}}\le \cosh\pr{1}\frac{2}{2k+1} +2\cos\pr{\frac{a}{2}}\\
    < \frac{1-2\cos\pr{\frac{a}{2}}}{2}+2\cos\pr{\frac{a}{2}} =\frac{1+2\cos\pr{\frac{a}{2}}}{2} 
\end{multline*}
hence,
\begin{equation}\label{eq: cos exp bound}
    \abs{2\cos\pr{\frac{z}{2}}}^{2k+1} < \pr{\frac{1+2\cos\pr{\frac{a}{2}}}{2} }^{2k+1}.
\end{equation}
If $\Im z = \pm\frac{2}{2k+1}$, then by \eqref{eq: abs(sin)} we have
\begin{multline*}
    \abs{2i\sin\pr{\frac{2k+1}{2}z}}^2 = 2\cosh\pr{\pr{2k+1}\frac{2}{2k+1}} -2\cos\pr{\pr{2k+1}\Re z} \\\ge 2\cosh\pr{2}-2>1,
\end{multline*}
and by \eqref{eq: cos exp bound} we have \[\abs{2\cos\pr{\frac{z}{2}}}^{2k+1}<1<\abs{2i\sin\pr{\frac{2k+1}{2}z}}.\]
If $\Re z = \frac{2\pi j}{2k+1}\pm\frac{\pi}{8k+4}$, then by \eqref{eq: abs(sin)} we have \[\abs{2i\sin\pr{\frac{2k+1}{2}z}}^2 = 2\cosh\pr{\pr{2k+1}z} -2\cos\pr{2\pi j \pm\frac{\pi}{4}} \ge 2-\sqrt{2}.\]
By \eqref{eq: cos exp bound} we get that \[\abs{2\cos\pr{\frac{z}{2}}}^{2k+1}<\pr{\frac{1+2\cos\pr{\frac{a}{2}}}{2}}^{2k+1}<\sqrt{2-\sqrt{2}}<\abs{2i\sin\pr{\frac{2k+1}{2}z}}.\]
The function $z\mapsto2i\sin\pr{\frac{2k+1}{2}z}$ has one simple zero in $R_{k,j}$, so by Rouch{\'e}'s theorem there exists one simple zero for $f_k$ in $R_{k,j}$. Therefore, there exists a unique $z_j\in R_{k,j}$ such that $f_{k}\pr{z_{j}} =0$. Hence,
\[2i\sin\pr{\frac{2k+1}{2}\pr{z_j-\frac{2\pi j}{2k+1}}}=2i\sin\pr{\frac{2k+1}{2}z_j} = -\pr{2\cos\pr{\frac{z}{2}}}^{2k+1}.\]
Observe that  \begin{equation}\label{eq: sin inv lip}
        \abs{\sin z - \sin w} \ge \frac{\sqrt{2}}{2}\abs{z-w}
    \end{equation}
    for all $z,w\in \prs{\Re\zeta\in\prb{-\frac{\pi}{4},\frac{\pi}{4}}}$ (see appendix for the proof of \eqref{eq: sin inv lip}), since $\frac{2k+1}{2}\pr{z_j-\frac{2\pi j}{2k+1}} \in \pr{-\frac{\pi}{4},\frac{\pi}{4}}$, we have 
\[\abs{2i\sin\pr{\frac{2k+1}{2}\pr{z_j-\frac{2\pi j}{2k+1}}}}\ge \frac{2k+1}{\sqrt{2}}\abs{z_j-\frac{2\pi j}{2k+1}},\]
and by \eqref{eq: cos exp bound} we have 
\begin{multline*}
    \abs{z_j-\frac{2\pi j}{2k+1}} \le \frac{\sqrt{2}}{2k+1}\abs{2\cos\pr{\frac{z_j}{2}}}^{2k+1}  \\<  \frac{\sqrt{2}}{2k+1}\pr{\frac{1+2\cos\pr{\frac{a}{2}}}{2} }^{2k+1} \le \frac{\sqrt{2}}{2k+1}e^{-ck}.\qedhere
\end{multline*}
\end{proof}

\subsection{Conformal properties of the modular lambda function}
Inspired by the methods of Bonk in \cite{bonk}, we utilize the fact that $\lambda$ commutes with some M{\"o}bius transformation to compute the images of circular arcs under $\lambda$. By Lemma \ref{lemma: rho to rho} 
    \[\lambda(\rho) = \rho=e^{\frac{i\pi}{3}}.\]
\begin{lemma}\label{lemma: lambda unit to line}
    For any $\tau\in \HH\cap\T$, we have $\Re\pr{\lambda\pr{\tau}}=\frac{1}{2}$. Furthermore, the function $\gamma\pr{\theta}=\lambda\pr{e^{i\theta}}$ parametrizes the line $\Re\pr{z}=\frac{1}{2}$ in a downward orientation; In particular, $\lambda\pr{\cU} = \cL_{\rho}$ and $\lambda\pr{\cU^*} = \overline{\cL_{\rho}}$ (see Figure \ref{fig: unit disk under lambda}).
\end{lemma}
\begin{proof}
For any $\tau\in\HH$, we have $\overline{\lambda\pr{\tau}} = \lambda\pr{-\bar{\tau}}$ and $\lambda\pr{-\frac{1}{\tau}}=1-\lambda\pr{\tau}$. Let $\tau\in \HH$ such that $\abs{\tau}=1$. Since $\frac{-1}{\tau}=-\bar{\tau}$, we have 
\[2\Re\pr{\lambda\pr{\tau}} = \lambda\pr{\tau}+\overline{\lambda\pr{\tau}}=\lambda\pr{\tau}+\lambda\pr{-\bar{\tau}}=\lambda\pr{\tau}+\lambda\pr{\frac{-1}{\tau}} = \lambda\pr{\tau}+1-\lambda\pr{\tau}=1.\]
Following $\lambda\pr{1}=\infty$ and Lemma \ref{lemma: rho to rho} we have $\lambda\pr{{\rho}}=\rho$, as $\theta$
increases from $0$ to $\frac{\pi}{3}$ (i.e.,\ $e^{i\theta}$ travels from $1$ to $\rho$) $\lambda$ travels from $\infty$ to $\rho$ on the line $\Re\pr{z}=\frac{1}{2}$. Recall that the interior of the hyperbolic triangle with angles $0$ whose vertices are $0$, $1$, and $i\infty$ is mapped to the upper-half plane under $\lambda$. Thus, by continuity, $\gamma\pr{\left(0,\frac{\pi}{3}\right]
}=\cL_{\rho}$ and the orientation is as desired. As for the rest of the line, it follows from continuity.
\end{proof}

\begin{figure}[ht]
    \centering
    \begin{subfigure}[b]{0.4\textwidth}
    \begin{tikzpicture}[scale = 1.3] 
    \definecolor{newblue}{RGB}{26, 128, 187};
    \definecolor{newred}{RGB}{193, 39, 45};
    \draw[thick,->] (-1.5,0) -- (1.5,0) node[right] {};
    \draw[fill=gray!30,draw =white] (-1,2.6) -- (-1,0) arc (180:0:0.5) arc (180:0:0.5) -- (1,2.6);
    \draw[line width=0.2mm] (-1,0) -- (-1,2.6);
    \draw[line width=0.15mm,dashed] (1,0) -- (1,2.6);
    \draw[line width=0.2mm] (0,0) arc (0:180:0.5);
    \draw[line width=0.15mm,dashed] (1,0) arc (0:180:0.5);
    \draw[line width=0.2mm, newred] (-1,0) arc (180:0:1);
    \draw[thin, newred,->] (80:0.9) arc (80:100:0.9);
    \filldraw (0,0) circle (0.75pt) node[below] {$0$};
    \filldraw (1,0) circle (0.75pt) node[below right] {$1$};
    \filldraw (-1,0) circle (0.75pt) node[below left] {$-1$};
    \filldraw (-0.5,0) circle (0.75pt) node[below] {$-\frac{1}{2}$};
    \filldraw (0.5,0) circle (0.75pt) node[below] {$\frac{1}{2}$};
    \filldraw (60:1) circle (0.75pt) node[above right] {$\rho$};
    \filldraw (120:1) circle (0.75pt) node[above left] {$-\bar{\rho}$};
    \filldraw (0,1) circle (0.75pt) node[above] {$i$};
    \node at (-60:1)  [below]{$\phantom{-\rho}$};
    \node at (2.4,0.5)  [left]{$\xrightarrow[\phantom{--}]{\lambda}$};
    \end{tikzpicture}
    \end{subfigure}
    \begin{subfigure}[b]{0.4\textwidth}
    \begin{tikzpicture}[scale = 1.3] 
    \definecolor{newblue}{RGB}{26, 128, 187};
    \definecolor{newred}{RGB}{193, 39, 45};
    \draw[thick,->] (-1.5,0) -- (2.5,0) node[right] {};
    \draw[loosely dashed] (-1,0) arc (180:-180:1);
    \draw[loosely dashed] (0,0) arc (180:-180:1);
    \draw[line width=0.2mm, newred] (0.5,-1.2) -- (0.5,2.6);
    \draw[thin, newred, ->] (0.35,0.6) -- (0.35,0.2);
    \filldraw (2,0) circle (0.75pt) node[below right] {$2$};
    \filldraw (1,0) circle (0.75pt) node[below right] {$1$};
    \filldraw (-1,0) circle (0.75pt) node[below left] {$-1$};
    \filldraw (0,0) circle (0.75pt) node[below left] {$0$};
    \filldraw (-60:1) circle (0.75pt) node[below right] {$\overline{\rho}$};
    \filldraw (60:1) circle (0.75pt) node[above right] {${\rho}$};
    \filldraw (0.5,0) circle (0.75pt) node[below right] {$\frac{1}{2}$};
    \end{tikzpicture}
    \end{subfigure}
    \caption{The mapping of the unit circle under $\lambda$.}
    \label{fig: unit disk under lambda}
\end{figure}

\begin{lemma}\label{lemma: lambda line to arc}
     For any $\tau\in \HH\cap\prs{\Re\pr{\tau}=\frac{1}{2}}$, we have $\abs{\lambda\pr{\tau}+1}=1$ with $\lambda\pr{\tau}\in\HH$. Furthermore, the function $\gamma\pr{t}=\lambda\pr{\frac{1}{2}+it}$ parametrizes the semi-circle $\prs{\abs{z-1}=1}\cap\HH$ in a clockwise orientation; in particular, the line $\cL_{\rho}$ is mapped to the arc $\cC$ (see Figure \ref{fig: lambda conform 2}).
\end{lemma}

\begin{proof}
    The M{\"o}bius transformation $\tau\mapsto\frac{1}{1-\tau}$ maps the arc $\prs{e^{i\varphi}:\varphi\in\left(0,\frac{\pi}{3}\right]}$ to the line $\cL_{\rho}$ in a downward orientation, and maps the line $\cL_{\rho}$ to the arc $\cC = \prs{e^{i\varphi}:\varphi\in\left[\frac{2\pi}{3},\pi\right)}$ in a counter-clockwise orientation. Let $t\ge \frac{\sqrt{3}}{2}$, then there exists $\varphi\in\left(0,\frac{\pi}{3}\right]$ such that $\frac{1}{2}+it = \frac{1}{1-e^{i\varphi}}$. Using the transformation formula $\lambda\pr{\frac{1}{1-\tau}} = \frac{1}{1-\lambda\pr{\tau}}$,   we have
    \[\lambda\pr{\frac{1}{2}+it}=\lambda\pr{\frac{1}{1-e^{i\varphi}}} = \frac{1}{1-\lambda\pr{e^{i\varphi}}}\in \cC.\]
    As $t$ increases, $\varphi$ decreases and $\Im\pr{\lambda\pr{e^{i\varphi}}}$ increases. Hence, $\lambda\pr{\frac{1}{2}+it}$ parametrizes the arc $\cC$ in a clockwise orientation. 
\end{proof}

\begin{figure}[ht]
    \centering
    \begin{subfigure}[b]{0.4\textwidth}
    \begin{tikzpicture}[scale = 1.3] 
    \definecolor{newblue}{RGB}{26, 128, 187};
    \definecolor{newred}{RGB}{193, 39, 45};
    \draw[thick,->] (-1.5,0) -- (1.5,0) node[right] {};
    \draw[fill=gray!30,draw =white] (-1,2.6) -- (-1,0) arc (180:0:0.5) arc (180:0:0.5) -- (1,2.6);
    \draw[line width=0.2mm] (-1,0) -- (-1,2.6);
    \draw[line width=0.15mm,dashed] (1,0) -- (1,2.6);
    \draw[line width=0.2mm] (0,0) arc (0:180:0.5);
    \draw[line width=0.15mm,dashed] (1,0) arc (0:180:0.5);
    \draw[line width=0.2mm,newblue] (-0.5,0.5) -- (-0.5,2.6);
    \draw[thin, newblue, ->] (-0.35,0.9) -- (-0.35,1.3);
    \draw[thin, newred, ->] (0.35,1.3) -- (0.35,0.9);
    \draw[line width=0.2mm,newred] (0.5,0.5) -- (0.5,2.6);
    \filldraw (0,0) circle (0.75pt) node[below] {$0$};
    \filldraw (1,0) circle (0.75pt) node[below right] {$1$};
    \filldraw (-1,0) circle (0.75pt) node[below left] {$-1$};
    \filldraw (-0.5,0) circle (0.75pt) node[below] {$-\frac{1}{2}$};
    \filldraw (0.5,0) circle (0.75pt) node[below] {$\frac{1}{2}$};
    \filldraw (60:1) circle (0.75pt) node[above right] {$\rho$};
    \filldraw (120:1) circle (0.75pt) node[above left] {$-\bar{\rho}$};
    \filldraw (0.5,0.5) circle (0.75pt) node[right] {\small $\frac{1+i}{2}$};
    \filldraw (-0.5,0.5) circle (0.75pt) node[left] {\small $\frac{-1+i}{2}$};
    \node at (-60:1)  [below]{$\phantom{-\rho}$};
    \node at (2.4,0.5)  [left]{$\xrightarrow[\phantom{--}]{\lambda}$};
    \end{tikzpicture}
    \end{subfigure}
    \begin{subfigure}[b]{0.4\textwidth}
    \begin{tikzpicture}[scale = 1.3] 
    \definecolor{newblue}{RGB}{26, 128, 187};
    \definecolor{newred}{RGB}{193, 39, 45};
    \draw[thick,->] (-1.5,0) -- (2.5,0) node[right] {};
    \draw[loosely dashed] (0.5,-1.2) -- (0.5,2.6);
    \draw[loosely dashed] (-1,0) arc (180:-180:1);
    \draw[line width=0.2mm, newred] (0,0) arc (180:0:1);
    \draw[thin, newred,->] ([shift=(100:0.9)]1,0) arc (100:80:0.9);
    \draw[line width=0.2mm, newblue] (2,0) arc (0:-180:1);
    \draw[thin, newblue,->] ([shift=(-80:0.9)]1,0) arc (-80:-100:0.9);
    \filldraw (2,0) circle (0.75pt) node[below right] {$2$};
    \filldraw (1,0) circle (0.75pt) node[below right] {$1$};
    \filldraw (-1,0) circle (0.75pt) node[below left] {$-1$};
    \filldraw (0,0) circle (0.75pt) node[below left] {$0$};
    \filldraw (-60:1) circle (0.75pt) node[below right] {$\overline{\rho}$};
    \filldraw (60:1) circle (0.75pt) node[above right] {${\rho}$};
    \filldraw (0.5,0) circle (0.75pt) node[below right] {$\frac{1}{2}$};
    \end{tikzpicture}
    \end{subfigure}
    \caption{The mapping of the lines $\prs{\frac{1}{2}+it:t>\frac{1}{2}}$ and $\prs{-\frac{1}{2}+it:t\ge\frac{1}{2}}$ under $\lambda$.}
    \label{fig: lambda conform 2}
\end{figure}

\begin{lemma}\label{lemma: lambda arcs to semi}
For any $\tau\in \HH\cap\prs{\abs{\tau-1}=1}$ we have $\abs{\lambda\pr{\tau}+1}=1$. Furthermore, the function $\gamma\pr{\theta}=\lambda\pr{1+e^{i\theta}}$ parameterize the semi-circle $\prs{\abs{z-1}=1}\cap \HH$ in a clockwise orientation; in particular, the arc $\cC$ is mapped to the arc $\cU$ (see Figure \ref{fig: lambda conform 3}).
\end{lemma}

\begin{proof}
    The proof is similar to the proofs of Lemma \ref{lemma: lambda unit to line} and Lemma \ref{lemma: lambda line to arc}, and is left for the reader as an exercise.
\end{proof}

\begin{figure}[ht]
    \centering
    \begin{subfigure}[b]{0.4\textwidth}
    \begin{tikzpicture}[scale = 1.3]
    \definecolor{newblue}{RGB}{26, 128, 187};
    \definecolor{newred}{RGB}{193, 39, 45};
    \draw[thick,->] (-1.5,0) -- (1.5,0) node[right] {};
    \draw[fill=gray!30,draw =white] (-1,2.6) -- (-1,0) arc (180:0:0.5) arc (180:0:0.5) -- (1,2.6);
    \draw[line width=0.2mm] (-1,0) -- (-1,2.6);
    \draw[line width=0.15mm,dashed] (1,0) -- (1,2.6);
    \draw[line width=0.2mm] (0,0) arc (0:180:0.5);
    \draw[line width=0.15mm,dashed] (1,0) arc (0:180:0.5);
    \draw[line width=0.2mm, newred] (0,0) arc (180:90:1);
    \draw[thin, newred,->] ([shift=(127:0.9)]1,0) arc (127:145:0.9);
    \draw[line width=0.2mm, newblue] (0,0) arc (0:90:1);
    \draw[thin, newblue,->] ([shift=(37:0.9)]-1,0) arc (37:55:0.9);
    \filldraw (0,0) circle (0.75pt) node[below] {$0$};
    \filldraw (1,0) circle (0.75pt) node[below right] {$1$};
    \filldraw (-1,0) circle (0.75pt) node[below left] {$-1$};
    \filldraw (-0.5,0) circle (0.75pt) node[below] {$-\frac{1}{2}$};
    \filldraw (0.5,0) circle (0.75pt) node[below] {$\frac{1}{2}$};
    \filldraw (60:1) circle (0.75pt) node[above] {$\rho$};
    \filldraw (120:1) circle (0.75pt) node[above] {$-\bar{\rho}$};
    \filldraw (1,1) circle (0.75pt) node[above right] {$1+i$};
    \filldraw (-1,1) circle (0.75pt) node[above left] {$-1+i$};
    \node at (-60:1)  [below]{$\phantom{-\rho}$};
    \node at (2.1,0.5)  [left]{$\xrightarrow[\phantom{--}]{\lambda}$};
    \end{tikzpicture}
    \end{subfigure}
    \begin{subfigure}[b]{0.4\textwidth}
    \begin{tikzpicture}[scale = 1.3] 
    \definecolor{newblue}{RGB}{26, 128, 187};
    \definecolor{newred}{RGB}{193, 39, 45};
    \draw[thick,->] (-1.5,0) -- (2.5,0) node[right] {};
    \draw[loosely dashed] (0.5,-1.2) -- (0.5,2.6);
    \draw[line width=0.2mm, newred] (-1,0) arc (180:0:1);
    \draw[thin, newred,->] (100:0.9) arc (100:80:0.9);
    \draw[line width=0.2mm, newblue] (1,0) arc (0:-180:1);
    \draw[thin, newblue,->] (-80:0.9) arc (-80:-100:0.9);
    \draw[loosely dashed] (0,0) arc (180:-180:1);
    
    \filldraw (2,0) circle (0.75pt) node[below right] {$2$};
    \filldraw (1,0) circle (0.75pt) node[below right] {$1$};
    \filldraw (-1,0) circle (0.75pt) node[below left] {$-1$};
    \filldraw (0,0) circle (0.75pt) node[below left] {$0$};
    \filldraw (60:1) circle (0.75pt) node[above right] {$\rho$};
    \filldraw (-60:1) circle (0.75pt) node[below right] {$\bar{\rho}$};
    \filldraw (0.5,0) circle (0.75pt) node[below right] {$\frac{1}{2}$};
    \end{tikzpicture}
    \end{subfigure}
    \caption{The mapping of the arcs $\prs{1+e^{i\varphi}:\varphi\in\pr{\frac{\pi}{2},\pi}}$ and $\prs{-1+e^{i\varphi}:\varphi\in\pr{0,\frac{\pi}{2}}}$ under $\lambda$.}
    \label{fig: lambda conform 3}
\end{figure}

As we wish to understand the height and limit distributions of the zeros, we require the following corollary of Lemma \ref{lemma: lambda line to arc}:
\begin{corr}\label{corr: functions on the arc and on the line}
    Define $\varphi:\left[\frac{\sqrt{3}}{2},\infty\right)\to\left[\frac{2\pi}{3},\pi\right)$ by 
     \begin{equation}\label{eq: arg}
         \varphi\pr{y} = -i\log\pr{\lambda\pr{\frac{1}{2}+iy}-1},\quad\forall y>\frac{\sqrt{3}}{2},
     \end{equation}
     where $\log$ is the branch of the logarithm satisfying  $-\frac{\pi}{2}<\Im\log\pr{z}<\frac{3\pi}{2}$. Then $\varphi$ is real-valued, strictly-increasing, onto, and differentiable. 
\end{corr}
\begin{remark}
    While trivial, Corollary \ref{corr: functions on the arc and on the line} above is essential to our understanding of the distribution in Theorem \ref{thm: Zeros Distrbution}.
\end{remark}§
\begin{proof}
    Let $y\ge\frac{\sqrt{3}}{2}$. Since $\lambda\pr{\rho} =\rho$ and $\lambda\pr{i\infty} =0$ and by Lemma \ref{lemma: lambda line to arc}, the line $\cL_{\rho}$ is mapped to the arc $\cC$. Thus, there exists $\theta\in\left[\frac{2\pi}{3},\pi\right)$ such that $\lambda\pr{\frac{1}{2}+iy} = 1+e^{i\theta}$, and therefore $\theta = -i\log\pr{\lambda\pr{\frac{1}{2}+iy}-1} = \varphi\pr{\theta}$. Hence, $\varphi$ is real-valued. By Lemma \ref{lemma: lambda line to arc}, it is strictly increasing and onto. Finally, it is trivially differentiable as a composition of differentiable functions. 
\end{proof}
\section{The even case}
Recall that in this case $4k=12\ell+k'$ with $k'\in\prs{0,4,8}$.
By Lemma \ref{lemma: lambda line to arc}, finding zeros of $\Theta_{\Gamma_{8k}}$ on the line $\cL_{\rho}$ is equivalent to finding zeros of $p_{2k}$ on the arc $\cC$.
\subsection{Proof of Theorem \ref{thm: Zeros}}\label{sec: zeros}
Before we can prove Theorem \ref{thm: Zeros}, we will need the following lemma:
\begin{lemma}\label{lemma: lambda approx}
    Let $\varphi$ be as defined in \eqref{eq: arg} in Corollary \ref{corr: functions on the arc and on the line}. As $y\to \infty$, we have 
    \[y=\frac{1}{\pi}\log\pr{\frac{16}{\sin\varphi\pr{y}}} + O\pr{e^{-3\pi y}}.\]
\end{lemma}
This lemma is the missing ingredient for the asymptotic formula for the heights of the zeros.

\begin{proof}
Let $y,\varphi$ such that $\lambda\pr{\frac{1}{2}+iy} =1+e^{i\varphi}$, i.e., $\varphi=\varphi\pr{y}$. Substituting $\tau=\frac{1}{2}+iy$ in the $q$-expansion of $\lambda$, we get:
\begin{multline*}
    1+e^{i\varphi} = \lambda\pr{\frac{1}{2}+iy} = \sum_{n=1}^{\infty} a\pr{n}e^{\frac{i\pi}{2}}e^{-\pi n y} \\
    = \sum_{n=1}^{\infty}\pr{-1}^n a\pr{2n}e^{-2\pi n y} +ie^{-\pi y}\sum_{n=0}^{\infty}\pr{-1}^n a\pr{2n+1}e^{-2\pi n y}.
\end{multline*}
Taking the imaginary part of both sides of the equation, we obtain 
\[\sin\varphi=e^{-\pi y}\sum_{n=0}^{\infty}\pr{-1}^n a\pr{2n+1}e^{-2\pi n y} = 16e^{-\pi y}+ O\pr{e^{-3\pi y}}\]
as $y\to \infty$. Therefore, 
\[y = \frac{1}{\pi}\log\pr{\frac{16}{\sin\varphi}\pr{1+O\pr{e^{-2\pi y}}}} = \frac{1}{\pi}\log\pr{\frac{16}{\sin\varphi}} + O\pr{e^{-2\pi y}}.\qedhere\]
\end{proof}

We are now ready to prove Theorem \ref{thm: Zeros}:
\begin{proof}[Proof of Theorem \ref{thm: Zeros}]
We have \[\Theta_{\Gamma_{8k}}\pr{\frac{1}{2}+iy} = \frac{\theta_{3}^{8k}\pr{\frac{1}{2}+iy}}{2}p_{2k}\pr{\lambda\pr{\frac{1}{2}+iy}} = \frac{\theta_{3}^{8k}\pr{\frac{1}{2}+iy}}{2}p_{2k}\pr{1+e^{i\varphi\pr{y}}}.\]
By Corollary \ref{corr: functions on the arc and on the line}, $\varphi$ is a bijection from $\left[\frac{\sqrt{3}}{2},\infty\right)$ to $\left[\frac{2\pi}{3},\pi\right)$. By Proposition \ref{prop: polyzeros}, there exist $\varphi_{k,1},\ldots,\varphi_{k,\ell}\in \prb{\frac{2\pi}{3},\pi}$ such that for any $1\le j\le \ell$, we have $p_{2k}\pr{1+e^{i\varphi_{k,j}}}$ and $\varphi_{k,j}\in\pr{\frac{\pi}{k}\pr{k-\ell +j-1},\frac{\pi}{k}\pr{k-\ell +j}}$. Denote $\tau_{j} = \lambda^{-1}\pr{1+e^{i\varphi_{k,\ell - j+1}}}$, then \[\Theta_{\Gamma_{8k}}\pr{\tau_{j}} = \frac{\theta_{3}^{8k}\pr{\tau_{j}}}{2}p_{2k}\pr{\lambda\pr{\tau_{j}}} = \frac{\theta_{3}^{8k}\pr{\tau_{j}}}{2}p_{2k}\pr{1+e^{i\varphi_{k,j}}}=0.\]
Therefore, we found $\ell$ inequivalent zeros in $\cL_{\rho}$, and thus $\Theta_{\Gamma_{8k}}$ has all of its zeros on $\cL_{\rho}$. Now, we have $\Im\tau_{1}=\varphi^{-1}\pr{\varphi_{k,\ell}}>\ldots > \varphi^{-1}\pr{\varphi_{k,1}} = \Im\tau_{\ell}$.

As $\varphi \to \pi^{-}$, we have $y=y\pr{\varphi}\to \infty$.
Hence, by Lemma \ref{lemma: lambda approx}, we have 
\[y=\frac{1}{\pi}\log\pr{\frac{16}{\sin\varphi}\pr{1+o\pr{1}}} = \frac{1}{\pi}\log\pr{\frac{16}{\sin\varphi}}+o\pr{1}\]
as $\varphi\to \pi^{-}$.
Using the fact that $\lim_{\varphi\to\pi}\frac{\sin\varphi}{\pi-\varphi}=1$, we have 
\[y= \frac{1}{\pi}\log\pr{\frac{16}{\sin\varphi}}+o\pr{1} = \frac{1}{\pi}\log(\frac{16}{\pi-\varphi})+o\pr{1}.\]
Let $m=o\pr{k}$ as $k\to\infty$. Since \[\varphi_{k,\ell-m+1} \in \pr{\pi-\frac{\pi m}{k},\pi-\frac{\pi\pr{m-1}}{k}},\]
we have that $\varphi_{k,\ell-m+1}\to\pi$ as $k\to\infty$.
Finally, since $\lim_{k\to\infty}\frac{\pi - \frac{\pi m}{k}}{\varphi_{\ell-m+1}} = 1$, we get \[\Im\tau_{m} = \frac{1}{\pi}\log\pr{\frac{16}{\pi-\pr{\pi-\frac{\pi m}{k}}}} + o\pr{1}= \frac{1}{\pi}\log\pr{\frac{16 k}{\pi m}} +o\pr{1}.\qedhere\]
\end{proof}.

\subsection{Proof of Theorem \ref{thm: Zeros Distrbution}}\label{sec: distribution} 
Our goal is to prove that the zeros are equidistributed on the line with respect to the density \[\varrho\pr{y} = \frac{3}{\pi}\varphi'\pr{y} = \frac{3}{\pi}\frac{\lambda'\pr{\frac{1}{2}+iy}}{\lambda\pr{\frac{1}{2}+iy}-1}.\] 
That is, to prove that for any $f\in C_{c}\pr{\cL_{\rho}}$, i.e., continuous with compact support, we have 
\[\frac{1}{\ell}\msum{\tau\in\cL_{\rho}}{\Theta_{\Gamma_{8k}}\pr{\tau}=0}{}f\pr{\tau}\xrightarrow{N\to\infty}\int_{\cL_{\rho}}f\pr{y}\varrho\pr{y}\dd{y}.\]
Recall the function $\varphi\pr{y} = -i\log(\lambda\pr{\frac{1}{2}+iy}-1)$ is a differential function on $\left[\frac{\sqrt{3}}{2},\infty\right)$.
Let $\tau_{1},\ldots,\tau_{\ell}\in\cL_{\rho}$ be the zeros of $\Theta_{\Gamma_{8k}}$. By Lemma \ref{corr: functions on the arc and on the line}, we have that $\varphi$ is real-valued and strictly increasing. Hence, by Proposition \ref{prop: q polyzeros}(ii), for all $\prb{a,b} \sub \left[\frac{\sqrt{3}}{2},\infty\right)$ we have: 
\begin{multline}
    \frac{\#\prs{1\le j\le \ell :\Theta_{\Gamma_{8k}}\pr{\tau_{j}}=0, \Im\pr{\tau_{j}}\in\prb{a,b}}}{\ell}
    \\ =  \frac{\#\prs{1\le j\le \ell:\varphi_{k,j}\in\prb{\varphi\pr{a},\varphi\pr{b}}}}{\ell} \xrightarrow[]{\ell\to\infty}\frac{3\pr{\varphi\pr{b}-\varphi\pr{a}}}{\pi}\\=\frac{3}{\pi}\int_{a}^{b}\varphi'\pr{y}dy = \int_{a}^{b}\varrho\pr{y}dy.
\end{multline}
For any $f\in C_{c}\pr{\cL_{\rho}}$, by approximating $f$ with linear combinations of indicators of intervals, we have 
\[\frac{1}{\ell}\sum_{j=1}^{\ell}f\pr{\tau_{j}}\xrightarrow[]{k\to\infty}\int_{\cL_{\rho}}f\pr{\tau}\varrho\pr{y}dy.\]
As $y\to\infty$, we have 
\[\frac{1}{\lambda\pr{\frac{1}{2}+iy}-1} = -1+O\pr{\lambda\pr{\frac{1}{2}+iy}} = -1+O\pr{e^{-\pi y}},\]
and $\lambda'\pr{\frac{1}{2}+iy}  = -16\pi e^{-\pi y} + O\pr{e^{-2\pi y}}$. Hence,
\[\varrho\pr{y} = \frac{3}{\pi}\frac{\lambda'\pr{\frac{1}{2}+iy}}{\lambda\pr{\frac{1}{2}+iy}-1} = 48e^{-\pi y}+ O\pr{e^{-2\pi y}},\]
which concludes the proof of Theorem \ref{thm: Zeros Distrbution}.\qed

\section{The odd case}\label{sec: odd case}
In this case, the forms $\Theta_{\Gamma_{8k+4}}$ are modular form of weight $4k+2$ for $\Gamma\pr{2}$. We write $4k+2 = 12\ell+k'$ with $k'\in\prs{6,10,14}$. 

\subsection{Proof of Theorem \ref{thm: Zeros odd case}} 
First, notice 
\begin{equation}\label{eq: coef odd}
    p_{2k+1}\pr{z} = 1+z^{2k+1}+\pr{z-1}^{2k+1} = 2 + \sum_{j=1}^{2k}\pr{-1}^{j}{\binom{2k+1}{j}z^{j}}.
\end{equation} Hence, the degree of $ p_{2k+1}$ is $2k$;  by Lemma \ref{lemma: simplicity}, for any $k\ge 1$, the zeros of $p_{2k+1}$ are simple and therefore there are $2k$ distinct zeros for $p_{2k+1}$. Since $\lambda$ is injective, there are $2k$ distinct zeros for $\Theta_{\Gamma_{8k+4}}$ in the fundamental domain. Use \eqref{eq: lambda rep} and \eqref{eq: coef odd}, and get
\[\Theta_{\Gamma_{8k+4}} = \frac{\theta_{3}^{4}}{2}\pr{2\theta_{3}^{8k}+ \sum_{j=1}^{2k}\pr{-1}^{j}{\binom{2k+1}{j}\theta_{3}^{8k-4j}\theta_{2}^{4j}}},\] since $\theta_{3}^{4}$ vanish at $i\infty$, so does $\Theta_{\Gamma_{8k+4}}$. 
By the valence formula \eqref{eq: lvl2 valence}:
\begin{equation*}
    \ord_{\infty}\pr{\Theta_{\Gamma_{8k+4}}}+\ord_\mathfrak{1}\pr{\Theta_{\Gamma_{8k+4}}}+\ord_\mathfrak{0}\pr{\Theta_{\Gamma_{8k+4}}}+\sum_{z\in\cF_{\lambda}}\ord_{z}\pr{\Theta_{\Gamma_{8k+4}}} = \frac{4k+2}{2} = 2k+1.
\end{equation*}
hence, the zero at $i\infty$ and the $2k$ distinct zeros in the fundamental domain are all simple and account for all of the zeros of $\Theta_{\Gamma_{8k+4}}$. Additionally, by Lemma \ref{lemma: lambda unit to line}, Lemma \ref{lemma: lambda line to arc}, and Lemma \ref{lemma: lambda arcs to semi} we have that $\lambda\pr{\cU} = \cL_{\rho}$, $\lambda\pr{\cU^*} = \overline{\cL_{\rho}}$, $\lambda\pr{\cL_{\rho}} = \cC$, $\lambda\pr{\cL_{\rho}^*} = \overline{\cC}$, $\lambda\pr{\cC} = \cU$ and $\lambda\pr{\cC^{*}}=\overline{\cU}$. Hence, any $\tau\in\cF_{\lambda}$ on one of the geodesics above is a zero of $\Theta_{\Gamma_{8k+4}}$ if and only if $\lambda\pr{\tau}$ is a zero of $p_{2k+1}$ on the map of that geodesic. Together with Lemma \ref{lemma: thm 1.3 lemma} we get that there are $\ell$ zeros of $\Theta_{\Gamma_{8k+4}}$ on each of the geodesics $\cU$ and $\cU^*$, and no zeros on the geodesics  $\cL_{\rho}$, $\cL_{\rho}^*$, $\cC$, and $\cC^{*}$.\qed

\subsection{Proof of Theorem \ref{thm: zeros odd case decay}} Let $\alpha\in \pr{0,\frac{1}{3}}$, denote $a = \frac{2\pi}{3} + \pi\frac{1-3\alpha}{6}=\frac{5\pi-3\alpha\pi}{6}$ and $b=\pi -\pi\frac{1-3\alpha}{6}=\frac{5\pi+3\alpha\pi}{6}$. Denote $\cR = \prb{a,b}\times\prb{-1/3,1/3}$, then $\cR$ is compact and $1+e^{iz}\neq0,1$ for all $z\in \cR$. Therefore, the derivative of $g\pr{z} = \lambda^{-1}\pr{1+e^{iz}}$ is bounded as it is continuous on $\cR$ and $\cR$ is compact. Furthermore, $\cR$ is convex and thus $g$ is Lipschitz on $\cR$, with the implied constant depending only on $\alpha$.

Using Proposition \ref{prop: poly odd decay}, for $c_{\alpha} = -2\log\pr{\frac{1+2\cos\pr{\frac{5\pi-3\alpha\pi}{12}}}{2}}$, $k\gg1$, and any $j\in \Z$ such that $a+\frac{\pi}{8k+4}\le \frac{2\pi j}{2k+1}\le \pi$, there exists $z_j \in \prb{a,\pi}\times\prb{-1/3,1/3}$ such that \[\abs{z_j - \frac{2\pi j}{2k+1}}\le \frac{\sqrt{2}}{2k+1}e^{-c_{\alpha}k}.\]
In particular, for any $j\in \Z$ such that $a+\frac{\pi}{8k+4}\le \frac{2\pi j}{2k+1}\le b-\frac{\pi}{8k+4}$, there exists $z_j \in \prb{a,b}\times\prb{-1/3,1/3}$ such that \[\abs{z_j - \frac{2\pi j}{2k+1}}\le \frac{\sqrt{2}}{2k+1}e^{-c_{\alpha}k}.\]
Let $m$ be the number of such $z_j$. The number $m$ is the same as the number of $j$-s satisfying
\begin{multline*}
    \frac{2k+1}{2}\cdot\frac{5-3\alpha}{6}+\frac{1}{8}\\
    =\frac{2k+1}{2\pi}a+\frac{1}{8}\le j\le \frac{2k+1}{2\pi}b-\frac{1}{8} = \frac{2k+1}{2}\cdot\frac{5+3\alpha}{6}-\frac{1}{8},
\end{multline*}
The number of integers in a closed interval is greater than one less than the length of the interval, i.e. 
\begin{multline*}
    m\ge \pr{\frac{2k+1}{2}\cdot\frac{5+3\alpha}{6}-\frac{1}{8}}-\pr{\frac{2k+1}{2}\cdot\frac{5-3\alpha}{6}+\frac{1}{8}} -1 \\= \frac{2k+1}{2}\alpha -\frac{5}{4}\ge \alpha k-2.
\end{multline*}
We also have that $\tau = \lambda^{-1}\pr{1+e^{iz_j}} = g\pr{z_j}$ is a zero of $\Theta_{\Gamma_{8k+4}}$ for any $z_j$ as above. Let $j_k$ be the least integer so that $a+\frac{\pi}{8k+4}\le \frac{2\pi j_k}{2k+1}$.

For any $1\le j\le m$ denote $\tau_{j} = \lambda^{-1}\pr{1+e^{iz_{j_k +j-1}}} = g\pr{z_{j_k +j-1}}$. By Lemma \ref{lemma: lambda line to arc} we have that $\Re\lambda^{-1}\pr{1+e^{\frac{i2\pi \pr{j_k +j-1}}{2k+1}}}=\frac{1}{2}$. Recall that $g$ is Lipschitz, with the implied constant depending only on $\alpha$. Therefore,
\begin{align*}
    \abs{\Re\tau_{j} - \frac{1}{2}} & =\abs{\Re\lambda^{-1}\pr{1+e^{iz_{j_k +j-1}}} - \Re\lambda^{-1}\pr{1+e^{\frac{2\pi\pr{j_k +j-1}}{2k+1}}}} \\
    & \le \abs{\lambda^{-1}\pr{1+e^{iz_{j_k +j-1}}} - \lambda^{-1}\pr{1+e^{\frac{i2\pi \pr{j_k +j-1}}{2k+1}}}} \\ 
    & \lesssim_{\alpha}\abs{z_{j_{k}+j-1} - \frac{2\pi\pr{j_{k}+j-1}}{2k+1}}\\
    & \ll_{\alpha} k^{-1}e^{-c_{\alpha k}},
\end{align*}
which concludes the proof of Theorem \ref{thm: zeros odd case decay}.\qed
\appendix
\section{}

Here, we prove the inequality we introduced in \eqref{eq: sin inv lip}:
\begin{lemma}\label{lemma: sin inv lip}
    For any $z,w\in\prs{\Re\zeta \in\prb{-\frac{\pi}{4},\frac{\pi}{4}}}$, we have
    \[\abs{\sin z - \sin w} \ge \frac{\sqrt{2}}{2}\abs{z-w}\]
\end{lemma}
\begin{proof}
Indeed, we have 
\begin{multline}\label{eq: sin ineq}
    \abs{\sin z - \sin w} =\abs{\pr{w-z}\int_{0}^{1}\cos\pr{z+t\pr{w-z}}\dd{t}} \\\ge \abs{z-w}\abs{\int_{0}^{1}\Re\pr{\cos\pr{z+t\pr{w-z}}}\dd{t}}.
\end{multline}
Is is known that $\cos\pr{x+iy}=\cos x\cosh y-i\sin x\sinh y$ for any $x,y\in\R$. Hence,
\begin{align*}
\abs{\int_{0}^{1}\Re\pr{\cos\pr{z+t\pr{w-z}}}\dd{t}} & =\abs{\int_{0}^{1}\cos\pr{\Re z+t\Re\pr{w-z}}\cosh\pr{\Im z +t\Im\pr{w-z}}\dd{z}\dd{t}}\\
& = \int_{0}^{1}\cos\pr{\Re z+t\Re\pr{w-z}}\cosh\pr{\Im z +t\Im\pr{w-z}}\dd{z}\dd{t}\\
& \ge \int_{0}^{1}\frac{\sqrt{2}}{2}\cdot 1\dd{t}=\frac{\sqrt{2}}{2},
\end{align*}
where we used the fact that $\cosh\pr{p}\ge 1$ for all $p\in \R$ and $\cos\pr{\Re z+t\Re\pr{w-z}}\ge \frac{\sqrt{2}}{2}$ for all $t\in \pr{0,1}$.
Together with \eqref{eq: sin ineq}, we get \eqref{eq: sin inv lip}
\end{proof}

Here we provide a detailed proof of Lemma \ref{prop: General Jacobi formula}.
\begin{proof}
    Recall that $\Gamma_{n} = D_{n}\cup \pr{D_{n}+\delta_{n}}$, where $D_{n}$ and $\delta_{n}$ are defined as in \eqref{eq: Dn lattice def}.
    Consider its theta function: 
    \begin{equation}\label{eq: splitting gamma}
        \Theta_{\Gamma_{n}}\pr{\tau} = \sum_{x\in D_n\cup\pr{D_n +\delta_n}}q_{2}^{\norm{x}^2} = \sum_{x\in D_n}q_{2}^{\norm{x}^2}+\sum_{x\in D_n+\delta_n}q_{2}^{\norm{x}^2}
    \end{equation}
    We begin by evaluating the sum over $D_n$ on the right-hand side of \eqref{eq: splitting gamma}. By definition $x_{1}+\ldots+x_n$ is even for all $x=\pr{x_1,\ldots,x_n}\in D_n$. Therefore,
    \[\norm{x}^2 = x_{1}^2+\ldots+x_n^{2} =x_{1}+\ldots+x_{n}\equiv0\bmod{2},\]
    as $x^2 \equiv x\bmod{2}$ for every integer $x$ and consequently,
    \[\ind_{D_n}\pr{x} = \frac{1}{2}\pr{1+\pr{-1}^{\norm{x}^2}},\quad\forall x\in\Z^n,\]
    where $\ind_{D_n}$ denotes the indicator function of $D_n$. Hence,
    \[\sum_{x\in D_n}q_{2}^{\norm{x}^2} = \sum_{x\in\Z^n}\ind_{D_n}\pr{x}q_{2}^{\norm{x}^2}=\frac{1}{2}\sum_{x\in\Z^n}q_{2}^{\norm{x}^2}+\frac{1}{2}\sum_{x\in\Z^n}\pr{-1}^{\norm{x}^2}q_{2}^{\norm{x}^2}.\]
    By definition, we have 
    \[\sum_{x\in\Z^n}q_{2}^{\norm{x}^2} = \theta_{3}^n\pr{\tau},\] and similarly, 
    \[\sum_{x\in\Z^n}\pr{-1}^{\norm{x}^2}q_{2}^{\norm{x}^2}=\theta_{4}^{n}\pr{\tau}.\]
    Therefore, 
    \begin{equation}\label{eq: sum on lattice Dn}
        \sum_{x\in D_n}q_{2}^{\norm{x}^2} = \frac{1}{2}\pr{\theta_{3}^n\pr{\tau} + \theta_{4}^n\pr{\tau}}
    \end{equation}
    To evaluate the sum over the shift of $D_n$, we denote ${x}^{*} = \pr{-x_1,x_2,\ldots,x_n}$ for all $x=\pr{x_1,\ldots,x_n}\in\R^n$. We have $\norm{{x}^{*}}=\norm{x}$ for all $x\in\R^n$ and $D_n^* = D_n$, i.e., $x^* \in D_n$ if and only if $x\in D_n$, as changing the sign of an integer does not change its parity. Hence,
    \begin{multline*}
    \sum_{x\in D_{n}+\delta_n}q_{2}^{\norm{x}^2} = \sum_{{x\in D_{n}}}q_{2}^{\norm{x+\delta_n}^2}
    =\sum_{{x\in D_{n}}}q_{2}^{\norm{x+\delta_n}^2}\\=\sum_{{x\in D_{n}}}q_{2}^{\norm{x^* +\delta_n^*}^2} = \sum_{{x\in D_{n}}}q_{2}^{\norm{x +\delta_n^*}^2}=\sum_{{x\in D_{n}}}q_{2}^{\norm{x+e_{1}^* +\delta_n}^2},
    \end{multline*}
    where $e_{1}=\pr{1,0,\ldots,0}\in\R^n$. $D_n$ is of index $2$ in $\Z^n$ and since $e_{1}^*\notin D_n$, we have $\Z^n = D_{n}\cup \pr{e_{1}^* +D_n}$. Thus,
    \[ \sum_{x\in \Z^{n}}q_{2}^{\norm{x+\delta_n}^2} =\sum_{{x\in D_{n}}}q_{2}^{\norm{x+\delta_n}^2}+\sum_{{x\in D_{n}}}q_{2}^{\norm{x+e_{1}^* +\delta_n}^2}=2\sum_{x\in D_{n}+\delta_n}q_{2}^{\norm{x}^2}.\]
    By definition 
    \[\sum_{x\in D_{n}+\delta_n}q_{2}^{\norm{x}^2}=\frac{1}{2}\sum_{x\in \Z^{n}}q_{2}^{\norm{x+\delta_n}^2} = \frac{1}{2}\theta_{2}^{n}\pr{\tau}.\]
    Connecting the equation above with \eqref{eq: splitting gamma} and \eqref{eq: sum on lattice Dn} we get \eqref{eq: jacobi rep} as required. 
\end{proof}

\bibliography{refs}

@book {ConwaySloane,
    AUTHOR = {Conway, J. H. and Sloane, N. J. A.},
     TITLE = {Sphere packings, lattices and groups},
    SERIES = {Grundlehren der mathematischen Wissenschaften [Fundamental
              Principles of Mathematical Sciences]},
    VOLUME = {290},
   EDITION = {Third},
      NOTE = {With additional contributions by E. Bannai, R. E. Borcherds,
              J. Leech, S. P. Norton, A. M. Odlyzko, R. A. Parker, L. Queen
              and B. B. Venkov},
 PUBLISHER = {Springer-Verlag, New York},
      YEAR = {1999},
     PAGES = {lxxiv+703},
      ISBN = {0-387-98585-9},
   MRCLASS = {11H31 (05B40 11H06 20D08 52C07 52C17 94B75 94C30)},
  MRNUMBER = {1662447},
MRREVIEWER = {Renaud\ Coulangeon},
       DOI = {10.1007/978-1-4757-6568-7},
       URL = {https://doi.org/10.1007/978-1-4757-6568-7},
}

@article {SpherePackingViazovska,
    AUTHOR = {Viazovska, Maryna S.},
     TITLE = {The sphere packing problem in dimension 8},
   JOURNAL = {Ann. of Math. (2)},
  FJOURNAL = {Annals of Mathematics. Second Series},
    VOLUME = {185},
      YEAR = {2017},
    NUMBER = {3},
     PAGES = {991--1015},
      ISSN = {0003-486X,1939-8980},
   MRCLASS = {52C17 (11F03 11F06 11H31)},
  MRNUMBER = {3664816},
MRREVIEWER = {Rainer\ Schulze-Pillot},
       DOI = {10.4007/annals.2017.185.3.7},
       URL = {https://doi.org/10.4007/annals.2017.185.3.7},
}

@article{Smith,
author = {Smith, Henry John Stephen },
title = {{III. O}n the orders and genera of quadratic forms containing more than three indeterminates.—Second notice},
journal = {Proceedings of the Royal Society of London},
volume = {16},
number = {},
pages = {197-208},
year = {1868},
doi = {10.1098/rspl.1867.0036},

URL = {https://royalsocietypublishing.org/doi/abs/10.1098/rspl.1867.0036},
eprint = {https://royalsocietypublishing.org/doi/pdf/10.1098/rspl.1867.0036}
}

@article{KorkineZolotareff,
	author = {Korkine, A. and Zolotareff, G.},
	date = {1873/09/01},
	doi = {10.1007/BF01442795},
	id = {Korkine1873},
	isbn = {1432-1807},
	journal = {Mathematische Annalen},
	number = {3},
	pages = {366--389},
	title = {Sur les formes quadratiques},
	url = {https://doi.org/10.1007/BF01442795},
	volume = {6},
	year = {1873}}

@article{Freedman,
author = {Michael Hartley Freedman},
title = {{The topology of four-dimensional manifolds}},
volume = {17},
journal = {Journal of Differential Geometry},
number = {3},
publisher = {Lehigh University},
pages = {357 -- 453},
year = {1982},
doi = {10.4310/jdg/1214437136},
URL = {https://doi.org/10.4310/jdg/1214437136}
}

@article {HearingDrums,
    AUTHOR = {Kac, Mark},
     TITLE = {Can one hear the shape of a drum?},
   JOURNAL = {Amer. Math. Monthly},
  FJOURNAL = {American Mathematical Monthly},
    VOLUME = {73},
      YEAR = {1966},
    NUMBER = {4},
     PAGES = {1--23},
      ISSN = {0002-9890,1930-0972},
   MRCLASS = {57.50 (00.00)},
  MRNUMBER = {201237},
MRREVIEWER = {I.\ Stakgold},
       DOI = {10.2307/2313748},
       URL = {https://doi.org/10.2307/2313748},
}

@article{sdr,
  title={On the zeros of {Eisenstein} series},
  author={Rankin, Fenny K.C. and Swinnerton-Dyer, Henry P.F.},
  journal={Bulletin of the London Mathematical Society},
  volume={2},
  number={2},
  pages={169--170},
  year={1970},
  publisher={Oxford University Press}
}

@article{dukejen,
  title={On the zeros and coefficients of certain weakly holomorphic modular forms},
  author={Duke, William and Jenkins, Paul},
  journal={Pure and Applied Mathematics Quarterly},
  volume={4},
  number={4},
  pages={1327--1340},
  year={2008},
  publisher={International Press of Boston}
}

@article {RVY,
    AUTHOR = {Reitzes, Sarah and Vulakh, Polina and Young, Matthew P.},
     TITLE = {Zeros of certain combinations of {E}isenstein series},
   JOURNAL = {Mathematika},
  FJOURNAL = {Mathematika. A Journal of Pure and Applied Mathematics},
    VOLUME = {63},
      YEAR = {2017},
    NUMBER = {2},
     PAGES = {666--695},
      ISSN = {0025-5793,2041-7942},
   MRCLASS = {11F11},
  MRNUMBER = {3706602},
MRREVIEWER = {SoYoung\ Choi},
       DOI = {10.1112/S0025579317000110},
       URL = {https://doi.org/10.1112/S0025579317000110},
}

@article {rudnick2023,
    AUTHOR = {Rudnick, Ze\'{e}v},
     TITLE = {Zeros of modular forms and {F}aber polynomials},
   JOURNAL = {Mathematika},
  FJOURNAL = {Mathematika. A Journal of Pure and Applied Mathematics},
    VOLUME = {70},
      YEAR = {2024},
    NUMBER = {2},
     PAGES = {Paper No. e12244, 12},
      ISSN = {0025-5793,2041-7942},
   MRCLASS = {11F11},
  MRNUMBER = {4717387},
}

@book{serre,
	author = {Serre, J.P.},
	isbn = {9781468498844},
	lccn = {70190089},
	publisher = {Springer New York},
	series = {Graduate Texts in Mathematics},
	title = {A Course in Arithmetic},
	url = {https://books.google.co.il/books?id=8fPTBwAAQBAJ},
	year = {2012}}

@article{rankin,
  title={The zeros of certain {Poincar{\'e}} series},
  author={Rankin, Robert Alexander},
  journal={Compositio Mathematica},
  volume={46},
  number={3},
  pages={255--272},
  year={1982}
}

@article{gun,
  title={On the zeros of certain cusp forms},
  author={Gun, Sanoli},
  journal={Mathematical Proceedings of the Cambridge Philosophical Society},
  volume={141},
  number={2},
  pages={191--195},
  year={2006},
  organization={Cambridge University Press}
}

@article{rudnick2005,
  title={On the asymptotic distribution of zeros of modular forms},
  author={Rudnick, Ze{\'e}v},
  journal={International Mathematics Research Notices},
  volume={2005},
  number={34},
  pages={2059--2074},
  year={2005},
  publisher={OUP}
}

@article {kimmel,
    AUTHOR = {Kimmel, Noam},
     TITLE = {Asymptotic {Z}eros of {P}oincar\'e{} {S}eries},
   JOURNAL = {Int. Math. Res. Not. IMRN},
  FJOURNAL = {International Mathematics Research Notices. IMRN},
      YEAR = {2024},
    VOLUME = {2024},
    NUMBER = {21},
     PAGES = {13808--13826},
      ISSN = {1073-7928,1687-0247},
   MRCLASS = {99-06},
  MRNUMBER = {4819875},
       DOI = {10.1093/imrn/rnae224},
       URL = {https://doi.org/10.1093/imrn/rnae224},
}

@article{holowinsky2010mass,
  title={Mass equidistribution for {Hecke} eigenforms},
  author={Holowinsky, Roman and Soundararajan, Kannan},
  journal={Annals of mathematics},
  pages={1517--1528},
  year={2010},
  publisher={JSTOR}
}

@article {OnoKohnenBruinier,
    AUTHOR = {Bruinier, Jan H. and Kohnen, Winfried and Ono, Ken},
     TITLE = {The arithmetic of the values of modular functions and the
              divisors of modular forms},
   JOURNAL = {Compos. Math.},
  FJOURNAL = {Compositio Mathematica},
    VOLUME = {140},
      YEAR = {2004},
    NUMBER = {3},
     PAGES = {552--566},
      ISSN = {0010-437X,1570-5846},
   MRCLASS = {11F03 (11F11 11F33)},
  MRNUMBER = {2041768},
MRREVIEWER = {Matthew\ G.\ Boylan},
       DOI = {10.1112/S0010437X03000721},
       URL = {https://doi.org/10.1112/S0010437X03000721},
}

@book {123,
    AUTHOR = {Bruinier, Jan Hendrik and van der Geer, Gerard and Harder,
              G\"unter and Zagier, Don},
     TITLE = {The 1-2-3 of modular forms},
    SERIES = {Universitext},
      NOTE = {Lectures from the Summer School on Modular Forms and their
              Applications held in Nordfjordeid, June 2004},
 PUBLISHER = {Springer-Verlag, Berlin},
      YEAR = {2008},
     PAGES = {x+266},
      ISBN = {978-3-540-74117-6},
   MRCLASS = {11-06 (11F11)},
  MRNUMBER = {2385372},
       DOI = {10.1007/978-3-540-74119-0},
       URL = {https://doi.org/10.1007/978-3-540-74119-0},
}

@article{raveh,
	author = {Raveh, Roei},
	date = {2025/10/28},
	doi = {10.1007/s40993-025-00679-x},
	id = {Raveh2025},
	isbn = {2363-9555},
	journal = {Research in Number Theory},
	number = {4},
	pages = {96},
	title = {On the Zeros of the Miller Basis of Cusp Forms},
	url = {https://doi.org/10.1007/s40993-025-00679-x},
	volume = {11},
	year = {2025}}

@book{ChandEllipticFunc,
    AUTHOR = {Chandrasekharan, K.},
     TITLE = {Elliptic functions},
    SERIES = {Grundlehren der mathematischen Wissenschaften [Fundamental
              Principles of Mathematical Sciences]},
    VOLUME = {281},
 PUBLISHER = {Springer-Verlag, Berlin},
      YEAR = {1985},
     PAGES = {xi+189},
      ISBN = {3-540-15295-4},
   MRCLASS = {11F11 (11-02 33A25)},
  MRNUMBER = {808396},
MRREVIEWER = {Marvin\ I.\ Knopp},
       DOI = {10.1007/978-3-642-52244-4},
       URL = {https://doi.org/10.1007/978-3-642-52244-4},
}

@article{bonk,
      title={Conformal maps and critical points of {E}isenstein series}, 
      author={Mario Bonk},
      journal={arXiv preprint arXiv:2506.21471},
      year={2025},
      eprint={2506.21471},
      archivePrefix={arXiv},
      primaryClass={math.CV},
      url={https://arxiv.org/abs/2506.21471}, 
}

@book {Iwaniec,
    AUTHOR = {Iwaniec, Henryk},
     TITLE = {Spectral methods of automorphic forms},
    SERIES = {Graduate Studies in Mathematics},
    VOLUME = {53},
   EDITION = {Second},
 PUBLISHER = {American Mathematical Society, Providence, RI; Revista
              Matem\'atica Iberoamericana, Madrid},
      YEAR = {2002},
     PAGES = {xii+220},
      ISBN = {0-8218-3160-7},
   MRCLASS = {11F72 (11F12 11F37)},
  MRNUMBER = {1942691},
       DOI = {10.1090/gsm/053},
       URL = {https://doi.org/10.1090/gsm/053},
}

@book{DiamondShurman,
    AUTHOR = {Diamond, Fred and Shurman, Jerry},
     TITLE = {A first course in modular forms},
    SERIES = {Graduate Texts in Mathematics},
    VOLUME = {228},
 PUBLISHER = {Springer-Verlag, New York},
      YEAR = {2005},
     PAGES = {xvi+436},
      ISBN = {0-387-23229-X},
   MRCLASS = {11Fxx},
  MRNUMBER = {2112196},
MRREVIEWER = {Henri\ Darmon},
}

@article{Helou,
    AUTHOR = {Helou, Charles},
     TITLE = {Cauchy-{M}irimanoff polynomials},
   JOURNAL = {C. R. Math. Rep. Acad. Sci. Canada},
  FJOURNAL = {La Soci\'et\'e{} Royale du Canada. L'Acad\'emie des Sciences.
              Comptes Rendus Math\'ematiques (Mathematical Reports)},
    VOLUME = {19},
      YEAR = {1997},
    NUMBER = {2},
     PAGES = {51--57},
      ISSN = {0706-1994},
   MRCLASS = {11R09 (11C08 11R32)},
  MRNUMBER = {1491993},
MRREVIEWER = {Maurice\ Mignotte},
}

@article{Nanninga,
    AUTHOR = {Nanninga, Paul M.},
     TITLE = {Cauchy-{M}irimanoff and related polynomials},
   JOURNAL = {J. Aust. Math. Soc.},
  FJOURNAL = {Journal of the Australian Mathematical Society},
    VOLUME = {92},
      YEAR = {2012},
    NUMBER = {2},
     PAGES = {269--280},
      ISSN = {1446-7887,1446-8107},
   MRCLASS = {11C08},
  MRNUMBER = {2999158},
MRREVIEWER = {Min\ Sha},
       DOI = {10.1017/S1446788712000195},
       URL = {https://doi.org/10.1017/S1446788712000195},
}

@article{zilka,
      title={Moments of the zeros of Faber polynomials of the {M}iller basis}, 
      author={Adi Zilka},
      journal={arXiv preprint arXiv:2510.05737},
      year={2025},
      eprint={2510.05737},
      archivePrefix={arXiv},
      primaryClass={math.NT},
      url={https://arxiv.org/abs/2510.05737}, 
}

@article{Cornelissen,
      title={Geodesic clustering of zeros of Eisenstein series for congruence groups}, 
      author={Sebastián Carrillo Santana and Gunther Cornelissen and Berend Ringeling},
      journal={arXiv preprint arXiv:2509.16108},
      year={2025},
      eprint={2509.16108},
      archivePrefix={arXiv},
      primaryClass={math.NT},
      url={https://arxiv.org/abs/2509.16108}, 
}

@article {choi,
    AUTHOR = {Choi, SoYoung and Im, Bo-Hae},
     TITLE = {On the zeros of certain weakly holomorphic modular forms for
              {$\Gamma_0^+(2)$}},
   JOURNAL = {J. Number Theory},
  FJOURNAL = {Journal of Number Theory},
    VOLUME = {166},
      YEAR = {2016},
     PAGES = {298--323},
      ISSN = {0022-314X,1096-1658},
   MRCLASS = {11F03 (11F11)},
  MRNUMBER = {3486279},
MRREVIEWER = {Armin\ Straub},
       DOI = {10.1016/j.jnt.2016.02.008},
       URL = {https://doi.org/10.1016/j.jnt.2016.02.008},
}
\bibliographystyle{plain}
\end{document}